\documentclass{amsart}
\usepackage[utf8]{inputenc}
\usepackage{lmodern,amssymb,latexsym,mathtools,slashed,nicefrac,comment}
\usepackage[all]{xy}
\usepackage[dvipsnames]{xcolor}
\usepackage[numbers]{natbib}
\usepackage{enumitem}
\usepackage{hyperref}

\newtheorem{theorem}{Theorem}[section]
\newtheorem{lemma}[theorem]{Lemma}
\newtheorem{proposition}[theorem]{Proposition}
\newtheorem{corollary}[theorem]{Corollary}
\newtheorem{conjecture}[theorem]{Conjecture}
\newtheorem{question}[theorem]{Question}

\newtheorem{thmx}{Theorem}
\newtheorem{corx}[thmx]{Corollary}

\theoremstyle{definition}
\newtheorem{definition}[theorem]{Definition}
\newtheorem{example}[theorem]{Example}

\newtheorem{assumption}[theorem]{Assumption}
\newtheorem{geometric setting}[theorem]{Geometric setting}

\theoremstyle{remark}
\newtheorem{remark}[theorem]{Remark}

\numberwithin{equation}{section}

\newenvironment{texteqn}
{\begin{equation} \begin{minipage}{0.9\linewidth}}
{\end{minipage} \end{equation} \ignorespacesafterend}

\newcommand{\KO}{\textnormal{KO}}

\newcommand{\cal}{\mathcal}
\newcommand{\tensor}{\otimes}

\newcommand{\dd}{\mathrm{d}}

\newcommand{\abs}[1]{\left\lvert#1\right\rvert}

\newcommand{\norm}[1]{\left\|#1\right\|}

\newcommand{\defeq}{\mathrel{\vcentcolon=}}
\newcommand{\vertiii}[1]{{\left\vert\kern-0.25ex\left\vert\kern-0.25ex\left\vert #1 
    \right\vert\kern-0.25ex\right\vert\kern-0.25ex\right\vert}}

\renewcommand{\rm}{\mathrm}

\DeclareMathOperator{\rad}{rad}

\DeclareMathOperator{\id}{id}

\DeclareMathOperator{\dom}{dom}

\DeclareMathOperator{\RR}{\mathbb{R}}
\DeclareMathOperator{\ZZ}{\mathbb{Z}}
\DeclareMathOperator{\CC}{\mathbb{C}}

\DeclareMathOperator{\ind}{index}
\DeclareMathOperator{\width}{width}

\DeclareMathOperator{\relind}{rel-ind}

\DeclareMathOperator{\scal}{scal}
\DeclareMathOperator{\nscal}{\overline\scal}
\DeclareMathOperator{\supp}{supp}

\DeclareMathOperator{\Pop}{P}
\DeclareMathOperator{\D}{D}
\DeclareMathOperator{\Qop}{Q}
\DeclareMathOperator{\Rop}{R}

\DeclareMathOperator{\Top}{T}

\DeclareMathOperator{\Bop}{B}
\DeclareMathOperator{\Zop}{Z}
\DeclareMathOperator{\cc}{c}

\DeclareMathOperator{\dist}{dist}
\DeclareMathOperator{\spinor}{\slashed{\mathfrak S}}
\DeclareMathOperator{\cl}{\textnormal{Cl}}

\DeclareMathOperator{\spind}{\slashed{\mathfrak D}}

\begin{document}

\title[A long neck principle for Riemannian spin manifolds]{A long neck principle for Riemannian spin manifolds with positive scalar curvature}

\author{Simone Cecchini}
\address{Mathematisches Institut,
Georg-August-Universit\"at, 
G{\"o}ttingen,
Germany}
\email{simone.cecchini@mathematik.uni-goettingen.de}

\begin{abstract}
We develop index theory on compact Riemannian spin manifolds with boundary in the case when the topological information is encoded by bundles which are supported away from the boundary.
As a first application, we establish a ``long neck principle'' for a compact Riemannian spin $n$-manifold with boundary $X$, stating that if $\scal(X)\geq n(n-1)$ and there is a nonzero degree map into the sphere $f\colon X\to S^n$ which is strictly area decreasing, then the distance between the support of $\dd f$ and the boundary of $X$ is at most $\pi/n$.
This answers, in the spin setting and for strictly area decreasing maps, a question recently asked by Gromov.
As a second application, we consider a Riemannian manifold $X$ obtained by removing $k$ pairwise disjoint embedded $n$-balls from a closed spin $n$-manifold $Y$.
We show that if $\scal(X)>\sigma>0$ and $Y$ satisfies a certain condition expressed in terms of higher index theory, then the radius of a geodesic collar neighborhood of $\partial X$ is at most $\pi \sqrt{(n-1)/(n\sigma)}$.
Finally, we consider the case of a Riemannian $n$-manifold $V$ diffeomorphic to $N\times [-1,1]$, with $N$ a closed spin manifold with nonvanishing Rosenebrg index.
In this case, we show that if $\scal(V)\geq\sigma>0$, then the distance between the boundary components of $V$ is at most $2\pi \sqrt{(n-1)/(n\sigma)}$.
This last constant is sharp by an argument due to Gromov.
\end{abstract}

\maketitle



\section{Introduction and main results}
The study of manifolds with positive scalar curvature has been a central topic in differential geometry in recent decades.
On closed spin manifolds, the most powerful known obstruction to the existence of such metrics is based on the index theory for the spin Dirac operator.
Indeed, the Lichnerowicz formula~\cite{Lic63} implies that, on a closed spin manifold $Y$ with positive scalar curvature, the spin Dirac operator is invertible and hence its index must vanish. 

When $X$ is a compact Riemannian manifold with boundary of dimension at least three, it is well known by classical results of Kazdan and Warner~\cite{KW75,KW75ii,KWiii} that $X$ always carries a metric of positive scalar curvature.
In order to use topological information to study metrics of positive scalar curvature on $X$, we need extra geometric conditions.
When $X$ is equipped with a Riemannian metric with a product structure near the boundary, it is well known~\cite{APSI,APSII,APSIII} that the Dirac operator with global boundary conditions is elliptic.
This fact has been extensively used in the past decades to study metrics of positive scalar curvature in the spin setting.

The purpose of this paper is to systematically extend the spin Dirac operator technique to the case when the metric does not necessarily have a product structure near the boundary and the topological information is encoded by bundles supported away from the boundary.
As an application, we prove some metric inequalities with scalar curvature on spin manifolds with boundary, following the point of view recently proposed by Gromov.


\subsection{Some questions by Gromov on manifolds with boundary}\label{SS:some questions by Gromov}
Recall that a map of Riemannian manifolds $f\colon M\to N$ is called $\epsilon$-area contracting if $\norm{f^\ast\omega}\leq\epsilon\norm{\omega}$, for all two-forms $\omega\in \Lambda^2(N)$.
When $\epsilon\leq 1$, we say that $f$ is \emph{area decreasing}.
When $\epsilon<1$, we say that $f$ is \emph{strictly area decreasing}.

Let $(X,g)$ be a compact oriented $n$-dimensional Riemannian manifold with boundary and let $f\colon (X,g)\to(S^n,g_0)$ be a smooth area decreasing map, where $g_0$ denotes the standard round metric on the sphere.
The ``length of the neck'' of $(X,f)$ is defined as the distance between the support of the differential of $f$ and the boundary of $X$.
The \emph{long neck problem}~\cite[page~87]{Gro19} consists in the following question.


\begin{question}[Long Neck Problem]\label{Q:LNP}
What kind of a lower bound on $\scal_g$ and a lower bound on the ``length of the neck'' of $(X, f)$ would make $\deg(f) = 0$?
\end{question}


\begin{remark}
In this case, the topological obstruction is the existence of an area decreasing map $f\colon (X,g)\to(S^n,g_0)$ of nonzero degree.
The extra geometric information is given by the ``length of the neck'' of $(X, f)$ and the lower bound of $\scal_g$.
\end{remark}


\begin{remark}
More precisely, Gromov~\cite[page~87]{Gro19} conjectured the existence of a constant $c_n>0$, depending only on the dimension $n$ of the manifold $X$, such that
\begin{equation}\label{E:LNP1}
    \left[\scal_g\geq n(n-1)\right]\&\left[\dist(\supp(\dd f),\partial X)\geq c_n\right]\Rightarrow \deg(f)=0.
\end{equation}
The main motivation of this paper is to prove this inequality in the case when $X$ is spin.
\end{remark}


We will now review two conjectures recently proposed by Gromov, which are related to the long neck problem.
Let $Y$ be a closed $n$-dimensional manifold.
Let $X$ be the $n$-dimensional manifold with boundary obtained by removing a small $n$-dimensional ball from $Y$.
Observe that $X$ is a manifold with boundary $\partial X\cong S^{n-1}$.
Let $g$ be a Riemannian metric on $X$.
For $R>0$ small enough, denote by $B_R(\partial X)$ the geodesic collar neighborhood of $\partial X$ of width $R$.
Gromov proposed the following conjecture~\cite[Conjecture~D',~11.12]{Gro18}.


\begin{conjecture}\label{C:punctured torus}
Let $Y$ be a closed $n$-dimensional manifold such that $Y$ minus a point admits no complete metric of positive scalar curvature.
Let $X$ be the manifold with boundary obtained by removing a small $n$-dimensional ball from $Y$.
Let $g$ be a Riemannian metric on $X$ whose scalar curvature is bounded from below by a constant $\sigma>0$.
Then there exists a constant $c>0$ such that if there exists a geodesic collar neighborhood $B_R(\partial X)$ of width $R$, then
\begin{equation}
    R\leq\frac{c}{\sqrt\sigma}.
\end{equation}
\end{conjecture}


Let us now consider a second situation related to the long neck principle.
Let $N$ be a closed manifold.
A band over $N$ is a manifold $V$ diffeomorphic to $N\times [-1,1]$.
If $g$ is a Riemannian metric on $V$, we say that $(V,g)$ is a Riemannian band over $N$ and define the width of $V$ by setting
\begin{equation}
    \width(V)\defeq\dist(\partial_-V,\partial_+V),
\end{equation}
where $\partial_-V$ and $\partial_+V$ are the boundary components of $V$ corresponding respectively to $N\times\{-1\}$ and $N\times\{1\}$.
Recently, Gromov proposed the following conjecture~\cite[Conjecture~C,~11.12]{Gro18}.


\begin{conjecture}\label{C:Gromov band width}
Let $N$ be a closed manifold of dimension $n-1\geq 5$ which does not admit a metric of positive scalar curvature.
Suppose $V$ is a Riemannian band over $N$ whose scalar curvature is bounded from below by a constant $\sigma>0$.
Then
\begin{equation}\label{E:sharp band width}
    \width(V)\leq 2\pi\sqrt{\frac{n-1}{\sigma n}}.
\end{equation}
\end{conjecture}


\begin{remark}\label{R:optimality}
In general, one can ask whether, under the same hypotheses of Conjecture~\ref{C:Gromov band width}, there exists a constant $c_n$, depending only on the dimension $n$ of the manifold $N$, such that the inequality
\begin{equation}\label{E:weak band width}
    \width(V)\leq \frac{c_n}{\sqrt\sigma }
\end{equation}
holds.
Gromov proved~\cite[Optimality~of~$2\pi/n$,~page~653]{Gro18} that the constant
\[c_n=2\pi\sqrt{\frac{n-1}{ n}}\] 
is optimal.
\end{remark}


\subsection{Codimension zero obstructions}
Let $(X,g)$ be a compact $n$-dimensional Riemannian spin manifold with boundary whose scalar curvature is bounded from below by a constant $\sigma>0$.
The first main result of this paper consists in a ``long neck principle'' in this setting.
Our method is based on the analysis of the incomplete Riemannian manifold $X^\rm o=X\setminus\partial X$.
The topological information is encoded by a pair of bundles with metric connections $E$ and $F$ over $X^\rm o$ which have isomorphic typical fibers and are trivializable outside a compact submanifold with boundary $L\subset X^\rm o$.
Our topological invariant is given by the index of a twisted spin Dirac operator $\D_{L_D}^{E,F}$ on the double $L_\rm D$ of $L$, constructed using the pair $(E,F)$.

In order to relate this invariant to the geometry of the manifold $X$, we make use of extra data.
We use the distance function from the deleted boundary $\partial X$ of $X^\rm o$ to construct a rescaling function $\rho$ in such a way that the Dirac operator of $X^\rm o$, rescaled by the function $\rho$, is essentially self-adjoint.
We also make use of a potential, i.e. a smooth function $\phi\colon X^\rm o\to[0,\infty)$ which vanishes on $L$ and is locally constant in a neighborhood of the deleted boundary $\partial X$.
Using these extra data, we construct a Fredholm operator $\Pop_{\rho,\phi}^{E,F}$ on $X^\rm o$ whose index coincides with the index of $\D_{L_D}^{E,F}$.
A vanishing theorem for the operator $\Pop_{\rho,\phi}^{E,F}$ allows us to give conditions on $\scal_g$ and $\dist(K,\partial V)$ in such a way that the index of $\D_{L_\rm D}^{E,F}$ must vanish.
Our method can be regarded as an extension to a certain class of incomplete manifolds of the technique of Gromov and Lawson~\cite{GL80,GL83}.


\begin{thmx}\label{T:LNP}
Let $(X,g)$ be a compact $n$-dimensional Riemannian spin manifold with boundary.
Let $f\colon X\to S^n$ be a smooth strictly area decreasing map.
If $n$ is odd, we make the further assumption that $f$ is constant in a neighborhood of $\partial X$.
Suppose that the scalar curvature of $g$ is bounded from below by a constant $\sigma>0$.
Moreover, suppose that
\begin{equation}\label{E:ALN1}
    \scal_g\geq n(n-1)\qquad\text{on }\supp(\dd f)
\end{equation}
and
\begin{equation}\label{E:ALN2}
    \dist\bigl(\supp(\dd f),\partial X\bigr)>\pi\sqrt{\frac{n-1}{n\sigma}}.
\end{equation}
Then $\deg(f)=0$.
\end{thmx}


\begin{remark}
Theorem~\ref{T:LNP} answers Question~\ref{Q:LNP} when $X$ is spin and even dimensional and $f$ is strictly area decreasing.
The case when $f$ is area decreasing can be treated with a slight modification of the techniques presented in this paper and will be discussed in a separated paper.
\end{remark}


\begin{remark}
Condition~\eqref{E:ALN2} implies that $f$ is constant in a neighborhood of each connected component of $\partial X$ so that the degree of $f$ is well defined.
The extra assumption when $n$ is odd is needed, at least with the argument used in this paper, to reduce the odd-dimensional case to the even-dimensional case.
We believe it is possible to drop this extra assumption.
\end{remark}


\begin{remark}
It is an interesting question whether, in dimension at most eight, it is possible to drop the spin assumption from Theorem~\ref{T:LNP} by using the minimal hypersurface technique of Schoen and Yau~\cite{SY79}.
In fact, it is not clear whether this method can be used to approach the long neck problem, due to the difficulties, pointed out in~\cite{CS19}, arising when the minimal hypersurface technique is used to treat maps that are area contracting.
\end{remark}


We now consider a higher version of the long neck principle.
Let $Y$ be a closed $n$-dimensional spin manifold with fundamental group $\Gamma$.
There is a canonical flat bundle $\mathcal{L}_Y$ over $Y$, called the Mishchenko bundle of $Y$, whose typical fiber is $C^\ast\Gamma$, the maximal real group $C^\ast$-algebra of $\Gamma$.
The Rosenberg index~\cite{Ros83,Ros86a,Ros86b} of $Y$ is the class $\alpha(Y)\in \KO_n(C^\ast\Gamma)$, obtained as the index of the spin Dirac operator twisted with the bundle $\mathcal{L}_Y$.
Here, $\KO_n(C^\ast\Gamma)$ is the real $K$-theory of $C^\ast\Gamma$.
The class $\alpha(Y)$ is the most general known obstruction to the existence of metrics of positive scalar curvature on $Y$.
Denote by $\spind_{Y,\underline{C^\ast\Gamma}}$ the spin Dirac operator twisted with the bundle $\underline{C^\ast\Gamma}$, the trivial bundle on $Y$ with typical fiber $C^\ast\Gamma$.
We assume that
\begin{texteqn}\label{C:relative higher index}
    the Rosenberg index $\alpha(Y)$ does not coincide with the index of $\spind_{Y,\underline{C^\ast\Gamma}}$.
\end{texteqn}


\begin{remark}
From results of Hanke and Schick~\cite{HS06,HS07}, closed enlargeable spin manifolds satisfy Condition~\eqref{C:relative higher index}.
For the notion of enlargeable manifold, see~\cite[\S IV.6]{LM89}.
Examples of closed enlargeable manifolds are the $n$-torus $T^n$ and any closed spin manifold admitting a metric of nonpositive sectional curvature.
Moreover, if $M_1$ and $M_2$ are closed spin manifolds and $M_1$ is enlargeable, then the connected sum $M_1\# M_2$ is enlargeable as well.
This provides us with a large class of examples satisfying Condition~\eqref{C:relative higher index}.
For more details and examples of enlargeable manifolds, we refer the reader to~\cite[Section~5]{GL83} and~\cite[\S~IV.6]{LM89}.
\end{remark}


\begin{remark}
An example of a manifold which is not enlargeable and satisfies Condition~\eqref{C:relative higher index} is given by $T^4\times N$, with $N$ a $K3$ surface.
\end{remark}


\begin{remark}
Another interesting class of manifolds satisfying Condition~\eqref{C:relative higher index} consists in aspherical spin manifolds whose fundamental group satisfies the strong Novikov conjecture.
\end{remark}


We use Condition~\eqref{C:relative higher index} to establish a ``higher neck principle''.
Let $D_1,\ldots,D_N$ be pairwise disjoint disks embedded in $Y$.
Consider the compact manifold with boundary
\[
    X\defeq Y\setminus\bigl(D_1^\rm o\sqcup\ldots D_N^\rm o\bigr),
\]
where $D_j^\rm o$ is the interior of $D_j$.
Observe that the boundary of $X$ is the disjoint union $\partial X=S^{n-1}_1\sqcup\cdots\sqcup S^{n-1}_N$, where $S^{n-1}_j\defeq\partial D_j$.
If $g$ is a Riemannian metric on $X$, the \emph{normal focal radius} of $\partial X$, denoted by $\rad^\odot_g(\partial X)$, is defined as follows.
For $R>0$ small enough, denote by $B_R(S^{n-1}_j)$ the geodesic collar neighborhood of $S^{n-1}_j$ of width $R$.
Define $\rad^\odot_g(\partial X)$ as the supremum of the numbers $R>0$ such that there exist pairwise disjoint geodesic collar neighborhoods $B_R(S^{n-1}_1),\ldots,B_R(S^{n-1}_N)$.


\begin{thmx}\label{T:HLN}
Let $Y$, $\Gamma$ and $X$ be as above.
Suppose the Rosenberg index $\alpha(Y)$ does not coincide with the index of $\spind_{Y,\underline{C^\ast\Gamma}}$.
Moreover, suppose $g$ is a Riemannian metric on $X$ whose scalar curvature is bounded from below by a constant $\sigma>0$.
Then
\begin{equation}\label{E:HLN1}
    \rad^\odot_g(\partial X)\leq\pi\sqrt{\frac{n-1}{n\sigma}}.
\end{equation}
\end{thmx}


In view of Conjecture~\ref{C:punctured torus}, it is natural to consider, under Condition~\eqref{C:relative higher index}, the manifold $Y$ with $N$ points removed and ask whether it admits complete metrics of positive scalar curvature.


\begin{thmx}\label{T:complete metrics on torus}
Let $Y$ be a closed spin manifold with fundamental group $\Gamma$ and let $P_1,\ldots,P_N$ be distinct points in $X$.
Suppose the Rosenberg index $\alpha(Y)$ does not coincide with the index of $\spind_{Y,\underline{C^\ast\Gamma}}$.
Then the open manifold $M\defeq Y\setminus\left\{P_1,\ldots,P_N\right\}$ cannot carry any complete metric of positive scalar curvature.
\end{thmx}


\begin{remark}
This theorem can be thought of as a ``codimension zero'' version of~\cite[corollary~B]{Ce18} and is proved with similar methods.
Theorem~\ref{T:complete metrics on torus} can also be regarded as a ``higher version'' of~\cite[Theorem~1.1]{Zha19}.
\end{remark}


\begin{remark}
When $N=1$, Theorems~\ref{T:HLN} and~\ref{T:complete metrics on torus} imply that Conjecture~\ref{C:punctured torus} holds with constant $c=\pi\sqrt{(n-1)/n}$ for all closed $n$-dimensional spin manifolds satisfying Condition~\eqref{C:relative higher index}.
\end{remark}


\begin{remark}
When $Y$ is simply connected, Condition~\eqref{C:relative higher index} is vacuous and Theorems~\ref{T:HLN} and~\ref{T:complete metrics on torus} are vacuous as well.
The geometric interpretation of this fact could be related to the observation of Gromov~\cite[page~723]{Gro19} that Conjecture~\ref{C:punctured torus} is probably vacuous for simply connected manifolds.
\end{remark}


\subsection{Codimension one obstructions}
Let us now consider an $n$-dimensional Riemannian band $(V,g)$ over a closed spin manifold $N$.
Let $\partial_\pm V$ and $\width(V)$ denote the same objects as in Subsection~\ref{SS:some questions by Gromov}.
In this case, our obstruction is the Rosenberg index of the $(n-1)$-dimensional spin manifold $N$.
In analogy with the case of codimension zero obstructions, we consider the incomplete manifold $V^\rm o=V\setminus\partial V$ and fix a rescaling function $\rho$ and a potential $\psi$.
We also assume that $\psi$ is compatible with the band $V$.
This means that there exist constants $\lambda_-<0<\lambda_+$ such that $\psi=\lambda_-$ in a neighborhood of the deleted negative boundary component $\partial_-V$ and $\psi=\lambda_+$ in a neighborhood of the deleted positive boundary component $\partial_+V$.
We use these extra data to construct a Fredholm operator $\Bop_{\rho,\psi}$ on $V^\rm o$ whose index coincide with $\alpha(N)$.
From a vanishing theorem for the index of the operator $\Bop_{\rho,\psi}$, we deduce the following result.


\begin{thmx}\label{T:sharp band width}
Let $N$ be a closed $(n-1)$-dimensional spin manifold with fundamental group $\Gamma$.
Suppose the Rosenberg index $\alpha(N)\in\KO_n(C^\ast\Gamma)$ does not vanish.
Let $V$ be a Riemannian band over $N$ whose scalar curvature is bounded from below by a constant $\sigma>0$.
Then
\begin{equation*}
    \width(V)\leq 2\pi\sqrt{\frac{n-1}{\sigma n}}.
\end{equation*}
\end{thmx}


\begin{remark}
This theorem implies that conjecture~\ref{C:Gromov band width} holds for all closed spin manifolds with nonvanishing Rosenberg index.
\end{remark}

\begin{remark}
In view of Remark~\ref{R:optimality}, the inequality found in Theorem~\ref{T:sharp band width} is sharp.
\end{remark}


\begin{remark}
Zeidler~\cite[Theorem~1.4]{Zei19} recently proved that, under the same hypotheses of Theorem~\ref{T:sharp band width}, there exists a constant $c$, independent of $n$, such that Inequality~\eqref{E:weak band width} holds.
This constant is numerically close to $20.51$: see Remark~\cite[Remark~1.9]{Zei19}.
Therefore, it is not optimal or asymptotically optimal (the asymptotically optimal constant would be $2\pi$).
Theorem~\ref{T:sharp band width} strengthens~\cite[Theorem~1.4]{Zei19} with the optimal constant.
This answers a question asked by Zeidler: see~\cite[Remark~1.9]{Zei19}.
\end{remark}


\noindent Theorem~\ref{T:sharp band width} implies the following relevant case of Gromov's Conjecture~\ref{C:Gromov band width}.


\begin{corx}
Conjecture~\ref{C:Gromov band width} holds when $N$ is a closed simply connected manifold of dimension at least $5$.
\end{corx}


\begin{remark}
This corollary strengthens~\cite[corollary~1.5]{Zei19} with the optimal constant.
It follows from Theorem~\ref{T:sharp band width} by the same argument used in~\cite{Zei19} so we do not repeat it here.
\end{remark}


The paper is organized as follows.
In Section~\ref{S:additivity formula}, we prove a $K$-theoretic additivity formula for the index in the setting of manifolds complete for a differential operator.
In Section~\ref{S:admissible rescaling}, we study rescaled Dirac operators and prove a Lichnerowicz-type inequality in this situation.
In Section~\ref{S:GL}, we construct the operator $\Pop_{\rho,\phi}^{E,F}$ and prove a formula to compute its index.
In Section~\ref{S:LNP}, we prove a vanishing theorem for the operator $\Pop_{\rho,\phi}^{E,F}$ and use it to prove Theorem~\ref{T:LNP}, Theorem~\ref{T:HLN}, and Theorem~\ref{T:complete metrics on torus}.
Finally, in Section~\ref{S:BW} we construct the operator $\Bop_{\rho,\psi}$ and use it to prove Theorem~\ref{T:sharp band width}.


\subsection*{Acknowledgment}
I am very thankful to Thomas Schick for many enlightening discussions and suggestions.
I would also like to thank the anonymous referee for having pointed out two technical issues in a previous version of this article and for having helped improving the quality of the paper.


\section{A $K$-theoretic additivity formula for the index}\label{S:additivity formula}
This section is devoted to the analytical background of this paper.
In Subsection~\ref{SS:diff operators}, we recall some preliminary notions on differential operators acting on bundles of modules over $C^\ast$-algebras and fix notation.
In Subsection~\ref{SS:relative completeness}, we consider a differential operator $\Pop$ on a not necessarily complete Riemannian manifold $M$.
In order to ensure that $\Pop$ has self-adjoint and regular closure, we make use of the notion of completeness of $M$ for $\Pop$, developed by Higson and Roe~\cite{HR00} and extended to the $C^\ast$-algebra setting by Ebert~\cite{Ebe16}.
When $\Pop^2$ is uniformly positive at infinity, by results of Ebert~\cite{Ebe16} the closure of $\Pop$ is Fredholm and its index is well defined.
In Subsection~\ref{SS:cut-and-paste}, we extend to this slightly more general class of operators a $K$-theoretic additivity formula due to Bunke~\cite{Bun95}.


\subsection{Differential operators linear over $C^\ast$-algebras}\label{SS:diff operators}
Throughout this paper, $A$ denotes a complex unital $C^\ast$-algebra.
We will also consider the case when $A$ is endowed with a Real structure.
We are mostly interested in the following two types of Real $C^\ast$-algebras.
The first one is the Real Clifford algebra $\cl_{n,m}$: see~\cite[Section~1.2]{Sch93} and~\cite[page~4]{Ebe16} for details.
The second one is the maximal group $C^\ast$-algebra $C^\ast\Gamma$ associated to a countable discrete group $\Gamma$.
This is the completion of the group algebra $\CC[\Gamma]$ with respect to the maximal norm and is endowed with a canonical Real structure induced by complex conjugation: see~\cite[Section~1.1]{Ebe16} and~\cite[Definition~3.7.4]{HR00}.

For Hilbert $A$-modules $H$ and $H^\prime$, we denote by $\cal L_A(H,H^\prime)$ the space of adjointable operators from $H$ to $H^\prime$ and by $\cal K_A(H,H^\prime)$ the subspace of the compact ones.
We also use the notation $\cal L_A(H)\defeq \cal L_A(H,H)$ and $\cal K_A(H)\defeq \cal K_A(H,H)$.
For the properties of Hilbert $A$-modules and adjointable operators, we refer to~\cite{Lan95} and~\cite[Section~15]{WO93}.

Let $(M,g)$ be a Riemannian manifold.
Let $W$ be a bundle of finitely generated projective Hilbert $A$-modules with inner product on $M$ and let $\Pop\colon\Gamma(M;W)\to\Gamma(M;W)$ be a formally self-adjoint differential operator of order one.
If $W$ is $\ZZ_2$-graded, we require that the operator $\Pop$ is odd with respect to the grading.
If $A$ has a Real structure, we require that $W$ is a bundle of finitely generated projective Real Hilbert $A$-modules and the operator $\Pop$ is real, i.e. $\Pop\kappa( w)=\kappa(\Pop w)$ for all $w\in\Gamma(M;W)$, where $\kappa$ is the involution defining the Real structure.
For more details, we refer to~\cite[Sections~1.1~and~1.2]{Ebe16}.
We are mostly interested in the two types of operators described in the following examples.


\begin{example}\label{Ex:complex dirac}
Let $(M,g)$ be a Riemannian spin manifold and let $E$ be a Hermitian vector bundle over $M$ endowed with a metric connection.
Let $\slashed S_M$ and $\slashed\D_M$ be the associated complex spinor bundle and complex spin Dirac operator .
Denote by $\slashed \D_{M,E}\colon\Gamma(M;\slashed S_M\tensor E)\to \Gamma(M;\slashed S_M\tensor E)$ the operator $\slashed\D_M$ twisted with the bundle $E$. 
If $M$ is even dimensional, $\slashed S_M$ is $\ZZ_2$-graded and the operator $\slashed \D_{M,E}$ is odd with respect to the induced $\ZZ_2$-grading on $\slashed S_M\tensor E$.
If in addition $M$ is closed, the operator $\slashed \D_{M,E}$ defines a class $\ind\left(\slashed\D_{M,E}\right)$ in $K_0(\CC )=\ZZ$.
For more details on this construction, we refer to~\cite[\S II.5]{LM89}.
\end{example}


\begin{example}\label{Ex:real Dirac}
Let $(M,g)$ be an $n$-dimensional Riemannian spin manifold.
Let $E$ be a bundle of finitely generated projective Real Hilbert $A$-modules with inner product and metric connection on $M$.
Let $\spinor_M$ be the $\cl_{n,0}$-spinor bundle on $(M,g)$ with associated $\cl_{n,0}$-linear spin Dirac operator $\spind_M$.
The bundle $\spinor_M $ is endowed with a $\cl_{n,0}$-valued inner product and is equipped with canonical Real structure and $\ZZ_2$-grading.
Let $\spind_{M,E}\colon\Gamma(\spinor_M\tensor E)\to\Gamma(\spinor_M\tensor E)$ be the operator $\spind_M$ twisted with the bundle $E$.
The $\ZZ_2$-grading on $\spinor_M$ induces a $\ZZ_2$-grading on $\spinor_M\tensor E$ and the operator $\spind_{M,E}$ is odd with respect to this grading.
When $M$ is closed, the operator $\spind_{M,E}$ defines a class $\ind\left(\spind_{M,E}\right)\in\KO_n(A)$.
For more details, see~\cite[\S II.7]{LM89} and~\cite[Section~1]{Ebe16}.
For the background material on Dirac operators twisted with bundles of Hilbert $A$-modules, we refer to~\cite[Section~6.3]{Sch05}.
We finally recall a particular instance of this construction, which is relevant for the geometric applications of this paper.
Let $M$ be a closed $n$-dimensional spin manifold with fundamental group $\Gamma$.
Let $\mathcal L_\Gamma$ be the Mishchenko bundle over $M$.
The bundle $\mathcal L_\Gamma$ has typical fiber $C^\ast\Gamma$ and is equipped with a canonical flat connection.
The class $\ind\left(\spind_{M,\mathcal L_\Gamma}\right)\in\KO_n(C^\ast\Gamma)$ is called the \emph{Rosenberg index} of $M$ and is denoted by $\alpha(M)$.
For more details, see~\cite{Ros07} and~\cite{Sto02}.
\end{example}


\begin{remark}
To be precise, $\ind\left(\spind_{M,E}\right)$ is a class in $\KO_n(A_{\RR})$, where $A_{\RR}$ is the real $C^\ast$-algebra consisting of the fixed points of the involution of $A$.
With a slight abuse of notation, we denote a Real $C^\ast$-algebra and its fixed point algebra by the same symbol.
\end{remark}


\begin{remark}
The fixed point algebra of $C^\ast\Gamma$ with respect to the canonical involution is the maximal real $C^\ast$-algebra of $\Gamma$, which in this paper will be denoted by the same symbol.
\end{remark}


\subsection{Manifolds which are complete for a differential operator}\label{SS:relative completeness}
Let $(M,g)$ be a Riemannian manifold.
Let $ W\to M$ be a bundle of finitely generated projective Hilbert $A$-modules with inner product and let $\Pop\colon\Gamma(M;W)\to\Gamma(M;W)$ be a formally self-adjoint differential operator of order one.
We regard $\Pop$ as a symmetric unbounded operator on $L^2(M;W)$ with initial domain $\Gamma_\rm c(M;W)$.
We will now give a condition so that its closure $\bar \Pop\colon\dom(\bar \Pop)\to L^2(M; W)$ is self-adjoint and regular.
For the background material on unbounded operators on Hilbert $A$-modules and the notion of regularity, see~\cite{Lan95}.


\begin{definition}\label{D:coercive}
A \emph{coercive function} is a proper smooth function $h\colon M\to \RR$ which is bounded from below.
\end{definition}


\begin{definition}
We say that the pair $(M,\Pop)$ is \emph{complete}, or that $M$ is \emph{complete for $\Pop$}, if there exists a coercive function $h\colon M\to\RR$ such that the commutator $\left[\Pop,h\right]$ is bounded.
\end{definition}


\begin{remark}\label{R:completeness depends on symbol}
The notion of completeness of a manifold for an operator depends only on the principal symbol of the operator.
This means that if $(M,\Pop)$ is complete and $\Phi\colon W\to W$ is a fiberwise self-adjoint bundle map, then $(M,\Pop+\Phi)$ is also complete.
\end{remark}


\begin{remark}\label{R:compact perturbation of coercive functions}
Suppose $h$ is a coercive function on $M$ and $\hat h\colon M\to\RR$ is a smooth function coinciding with $h$ outside of a compact set.
Then $\hat h$ is a coercive function as well.
Moreover, $\left[\Pop,h\right]$ is bounded if and only if $[\Pop,\hat h]$ is bounded.
\end{remark}


The next theorem, due to Ebert, gives the wanted sufficient condition.
It is a generalization to operators linear over $C^\ast$-algebras of a result of Higson and Roe~\cite[Proposition~10.2.10]{HR00}.


\begin{theorem}[Ebert, {\cite[Theorem~1.14]{Ebe16}}]\label{T:HR}
If $(M,\Pop)$ is complete, then the closure of $\Pop$ is self-adjoint and regular.
\end{theorem}


Assume $(M,\Pop)$ is complete and denote the self-adjoint and regular closure of $\Pop$ by the same symbol.
Assume also there is a $\ZZ_2$-grading $W=W^+\oplus W^-$ and the operator $\Pop$ is odd with respect to this grading,
i.e. it is of the form
\begin{equation}\label{E:D is odd}
    \Pop=\begin{pmatrix}0&\Pop^-\\\Pop^+&0
    \end{pmatrix}\,,
\end{equation}
where $\Pop^\pm\colon\Gamma(M;W^\pm)\to\Gamma(M;W^\mp)$ are formally adjoint to one another.
Finally, assume $\Pop$ is elliptic.

To simplify the notation, in the remaining part of this section we set $H\defeq L^2(M;W)$.
We say that the operator $\Pop^2$ is \emph{uniformly positive at infinity} if there exist a compact subset $K\subset M$ and a constant $c>0$ such that
\begin{equation}\label{E:invertible at infty}
	\bigl<\Pop^2 w,w\bigr>\geq c\bigl< w,w\bigr>,\qquad w\in\Gamma_\rm c\bigl(M\setminus K;W|_{M\setminus K}\bigr).
\end{equation}
In this case, by~\cite[Theorem~2.41]{Ebe16} the operator $\Pop\left(\Pop^2+1\right)^{-1/2}\in\cal L_A(H)^\rm{odd}$ is Fredholm.
We denote its index by $\ind\left(\Pop\right)$.


In the next lemma, we collect some properties of the operator $\Pop$ that will be needed in the proof of the additivity formula.


\begin{lemma}\label{P:invertible at infinity}
The operator $\Pop^2+1+t^2$ is invertible for every $t\geq0$.
Moreover, $\bigl(\Pop^2+1+t^2\bigr)^{-1} $ is a positive element of $\mathcal{L}_A(H)$ and there is the absolutely convergent integral representation
\begin{equation}\label{E:integral representation}
	\bigl(\Pop^2+1\bigr)^{-1/2}=\frac{2}{\pi}\int_0^\infty (\Pop^2+1+t^2)^{-1}\,\dd t.
\end{equation}
Finally, we have the estimates
\begin{align}\label{E:R(t)}
    &\bigl\|(\Pop^2+1+t^2)^{-1}\bigr\|_{\mathcal{L}_A(H)}\leq (1+t^2)^{-1}\\\label{E:DR(t)1}
    &\bigl\|\Pop (\Pop^2+1+t^2)^{-1}\bigr\|_{\mathcal{L}_A(H)}\leq \frac{1}{2\sqrt{1+t^2}}\\\label{E:D^2R(t)1}
	&\bigl\|\Pop^2 \left(\Pop^2+1+t^2\right)^{-1}\bigr\|_{\mathcal{L}_A(H)}\leq 1
\end{align}
for all $t\geq 0$.
\end{lemma}


\begin{proof}
The first part of the lemma and Inequality~\eqref{E:R(t)} follow from~\cite[Proposition~1.21]{Ebe16}.
Inequalities~\eqref{E:DR(t)1} and~\eqref{E:D^2R(t)1} follow from Part~(2) of~\cite[Theorem~1.19]{Ebe16}.
\end{proof}


\subsection{Cut-and-paste invariance}\label{SS:cut-and-paste}
For $i=1,2$, let $M_i$ be a Riemannian manifold, let $W_i=W_i^+\oplus W_i^-$ be a $\ZZ_2$-graded bundle of finitely generated projective Hilbert $A$-modules with inner product and let $\Pop_i$ be an odd formally self-adjoint elliptic differential operator of order one.
We assume that $(M_i,\Pop_i)$ is complete and that $\Pop_i^2$ is uniformly positive at infinity so that its index is well defined.
Let $U_i\cup_{N_i}V_i$ be a partition of $M_i$, where $N_i$ is a closed separating hypersurface.
This means that $M_i=U_i\cup V_i$ and $U_i\cap V_i=N_i$.
We make the following assumption.


\begin{assumption}\label{A:cut-and-paste}
\emph{The operators coincide near the separating hypersurfaces}.
This means that there exist tubular neighborhoods $\cal U(N_1)$ and $\cal U(N_2)$ respectively of $N_1$ and $N_2$ and an isometry $\Gamma\colon\cal U(N_1)\to\cal U(N_2)$ such that $\Gamma|_{N_1}\colon N_1\to N_2$ is a diffeomorphism and $\Gamma$ is covered by a bundle isometry 
\[
    \widetilde \Gamma\colon W_1\big|_{\cal U(N_1)}\to W_2\big|_{\cal U(N_2)}\quad\text{so that}\quad
    \Pop_2=\widetilde \Gamma\circ\Pop_1\circ\widetilde \Gamma^{-1}
\]
in $\cal U(N_2)$.
\end{assumption}


\noindent This assumption allows us to do the following cut-and-paste construction.
Cut the manifolds $M_i$ and the bundles $W_i$ along $N_i$.
Use the map $\Gamma$ to interchange the boundary components and construct the Riemannian manifolds
\[
    M_3\defeq U_1\cup_NV_2\qquad\text{and}\qquad 
    M_4\defeq U_2\cup_NV_1,
\]
where $N\cong N_1\cong N_2$.
Moreover, using the map $\widetilde\Gamma$ to glue the bundles, we obtain $\ZZ_2$-graded bundles
\[
    W_3\defeq W_1\big|_{U_1}\cup_N W_2\big|_{V_2}\qquad\text{and}\qquad 
    W_4\defeq W_2\big|_{U_2}\cup_N W_1\big|_{V_1},
\]
and odd formally self-adjoint elliptic differential operators of order one $\Pop_3$ and $\Pop_4$.
Observe that, using Remark~\ref{R:compact perturbation of coercive functions}, the pairs $(M_3,\Pop_3)$ and $(M_4,\Pop_4)$ are complete and that the operators $\Pop_3^2$ and $\Pop_4^2$ are uniformly positive at infinity.
Therefore, the indices of $\Pop_3$ and $\Pop_4$ are well defined.
The next theorem is a slight generalization of~\cite[Theorem~1.2]{Bun95}.


\begin{theorem}\label{T:relative}
$\ind\left(\Pop_1\right)+\ind\left(\Pop_2\right)=\ind\left(\Pop_3\right)+\ind\left(\Pop_4\right)$.
\end{theorem}


\begin{proof}
Use the notation 
\[
    \cal H=H_1\oplus H_2\oplus H_3^\rm{op}\oplus H_4^\rm{op}\qquad\text{and}\qquad
    \cal F=F_1\oplus F_2\oplus F_3\oplus F_4,
\]
where $H_i= L^2(M_i;W_i)$, $F_i= P_i(P_i^2+1)^{-1/2}$ and $\Rop_i(t)=\bigl(\Pop^2_i+1+t^2\bigr)^{-1}$.
In order to prove the thesis, we need to show that $\ind\left(\cal F\right)=0$.

Pick cutoff functions $\chi_{U_i}$ and $\chi_{V_i}$ such that 
\[
    \supp\left(\chi_{U_i}\right)\subset U_i\cup\cal U(N_i)\qquad
    \supp\left(\chi_{V_i}\right)\subset V_i\cup\cal U(N_i)\qquad
    \chi_{U_i}^2+\chi_{V_i}^2=1.
\]
Moreover, we assume that $\chi_{U_1}=\chi_{U_2}$ and $\chi_{V_1}=\chi_{V_2}$ when restricted to $\cal U(N)\cong\cal U(N_1)\cong\cal U(N_2)$.
Multiplication by $\chi_{U_1}$ defines an operator $a\in\cal L_A(H_3,H_1)$.
Similarly, use the cutoff functions to define operators $b\in\cal L_A(H_1,H_4)$, $c\in\cal L_A(H_2,H_3)$, and $d\in\cal L_A(H_2,H_4)$.
Consider the operator
\[
    \cal X\defeq z\begin{pmatrix}0&0&-a^\ast&-b^\ast\\0&0&-c^\ast&d^\ast\\a & c &0&0\\ b &-d &0&0
    \end{pmatrix}\in\cal L_A(\cal H),
\]
where $z\in\cal L_A(H)$ is the $\ZZ_2$-grading.
As explained in~\cite[Subsection~3.1]{CB18} and in the proof of~\cite[Theorem~1.14]{Bun95}, in order to show that $\ind\left(\cal F\right)=0$, it suffices to show that $\cal X\cal F+\cal F\cal X\in\cal K_A(\cal H)$.
To this end, it is enough to verify the compactness of operators of the form $a^\ast F_3-F_1 a^\ast\in\cal L_A(H_3,H_1)$.

Let $\chi=\chi_{U_1}$ and let $\rho\in C^\infty_\rm c(U_i\cup\cal U(N_i))$ be such that $\rho\chi=\chi$.
Using Assumption~\ref{A:cut-and-paste}, the operators $\chi\Pop_3-\Pop_1\chi$ and $\left(\chi\Pop_1-\Pop_1\chi\right)\rho$ define the same element in $\cal L_A(H_3,H_1)$, that we denote by $[\Pop,\chi]$.
Using the integral representation~\eqref{E:integral representation} and the computations in~\cite[page~13]{Bun95}, we obtain
\begin{equation}\label{E:chiF3-F1chi compact}
    \chi F_3-F_1\chi=\frac{2}{\pi}\int_0^\infty\bigl(\chi \Pop_3\Rop_3(t)-\Pop_1\Rop_1(t)\chi\bigr)\,\dd t=\frac{2}{\pi}\int_0^\infty \Qop_{3,1}(t)\,\dd t,
\end{equation}
where
\begin{equation*}
    \Qop_{3,1}(t)\defeq -[\Pop,\chi]\Rop_3(t)+\Pop_1^2\Rop_1(t)[\Pop,\chi]\Rop_3(t)
    +\Pop_1\Rop_1(t)[\Pop,\chi]\Pop_3\Rop_3(t).
\end{equation*}
Using Inequalities~\eqref{E:R(t)},~\eqref{E:D^2R(t)1} and~\eqref{E:DR(t)1} and~\cite[Theorem~2.33 and~Remark~2.35]{Ebe16}, we deduce that the operator $\Qop_{3,1}(t)$ is compact and absolutely integrable.
By~\eqref{E:chiF3-F1chi compact}, $a^\ast F_3-F_1 a^\ast\in\cal K_A(H_3,H_1)$, which concludes the proof.
\end{proof}


\section{A rescaled Dirac operator}\label{S:admissible rescaling}
In this section, we present a general method to construct a complete pair on a Riemannian spin manifold.
Our method is based on rescaling the possibly twisted spin Dirac operator.
Moreover, we prove an estimate from below for the square of the rescaled twisted Dirac operator.
Finally, in order to obtain a slight improvement of this estimate, we extend to operators linear over $C^\ast$-algebras an inequality due to Friedrich~\cite[Thm.A]{Fr80} on closed manifolds and generalized by B\"ar~\cite[Theorem~3.1]{Bae09} to open manifolds.
This improvement will be used in Sections~\ref{S:LNP} and~\ref{S:BW} to obtain the factor $\sqrt{(n-1)/n}$ in Theorems~\ref{T:LNP},~\ref{T:HLN}, and~\ref{T:sharp band width}.
Even if we mostly focus on the spin case, all the results of this section hold with the obvious modifications for any operator of Dirac type.


\subsection{Admissible rescaling functions}\label{SS:admissible rescaling}
Let $(M,g)$ be a Riemannian manifold.
Let $ V\to M$ be a bundle of finitely generated projective Hilbert $A$-modules with inner product and let $\Zop\colon\Gamma(M; V)\to\Gamma(M; V)$ be a formally self-adjoint elliptic differential operator of order one such that
\begin{equation}\label{E:uniformly bounded symbol}
    \norm{\left[\Zop,\xi\right]_x}\leq \abs{\dd \xi_x},\qquad \xi\in C^\infty(M),\ x\in M.
\end{equation}
Here, $\norm{\left[\Zop,\xi\right]_x}$ is the norm of the adjointable map $[\Zop,\xi]_x\colon V_x\to V_x$.
For a function $\rho\colon M\to(0,\infty)$, define the rescaled operator $\Zop_\rho\colon\Gamma(M; V)\to\Gamma(M; V)$ as
\begin{equation}\label{E:rescaled operator}
    \Zop_\rho\defeq \rho\Zop\rho.
\end{equation}
Observe that $\Zop_\rho$ is a formally self-adjoint differential operator of order one and
\begin{equation}\label{E:rescaled symbol}
    \left[\Zop_\rho,\xi\right]=\rho^2 \left[\Zop,\xi\right],\qquad \xi\in C^\infty(M).
\end{equation}
Therefore, $\Zop_\rho$ is elliptic.


\begin{definition}\label{D:admissible rescaling}
A smooth function $\rho\colon M\to(0,1]$ is called an \emph{admissible rescaling function} for $M$ if there exists a coercive~(see Definition~\ref{D:coercive}) function $h$ such that $\rho^2\abs{\dd h}$ is in $L^\infty(M)$.
\end{definition}


\begin{remark}\label{R:coinciding at infty}
The property for a smooth function $\rho$ of being an admissible rescaling function depends only on its behaviour at infinity.
Moreover, suppose $\rho_1$, $\rho_2\colon M\to(0,1]$ are smooth functions such that $\rho_1$ is admissible and $\rho_2=b\rho_1$ outside of a compact set for some constant $b>0$.
Then $\rho_2$ is admissible as well.
\end{remark}


\begin{proposition}\label{P:adm->complete}
Let $\rho$ be an admissible rescaling function.
Then the pair $\bigl(M,\Zop_\rho\bigr)$ is complete.
\end{proposition}


\begin{proof}
Since $\rho$ is admissible, choose a coercive function $h$ such that $\rho^2\abs{\dd h}$ is in $L^\infty(M)$.
By~\eqref{E:uniformly bounded symbol} and~\eqref{E:rescaled symbol}, we deduce
\[
    \norm{\left[\Zop_\rho,h\right]v}\leq \norm{\rho^2\abs{\dd h}}_\infty\norm{v},\qquad v\in\Gamma_\rm c(M;V).\qedhere
\]
\end{proof}


\begin{remark}\label{R:complete manifolds}
When $(M,g)$ is a complete Riemannian manifold, the function $\rho=1$ is admissible and Proposition~\ref{P:adm->complete} implies the classical fact that a Dirac operator on $(M,g)$ is essentially self-adjoint.
\end{remark}


We now describe a method for constructing admissible rescaling functions on open Riemannian manifolds.
In Sections~\ref{S:LNP} and~\ref{S:BW}, we will use this method together with the geometry at infinity of the manifolds to construct complete pairs.


\begin{proposition}\label{P:admissible criterion}
Let $\tau\colon M\to(0,\infty)$ be a smooth function such that 
\begin{equation}\label{E:tau infty}
	\lim_{x\to\infty}\tau(x)=0
\end{equation}
and there exists a constant $c>0$ satisfying
\begin{equation}\label{E:dtau}
	\abs{\dd\tau_x}\leq c,\qquad x\in M.
\end{equation}
Suppose $\gamma_\alpha\colon (0,\infty)\to (0,1]$ is a smooth function such that $\gamma_\alpha(t)=t^\alpha$ for $t$ near $0$.
Then $\rho_\alpha\defeq \gamma_\alpha\circ \tau$ is an admissible rescaling function for all $\alpha\geq 1/2$.
\end{proposition}


\begin{proof}
Observe, using~\eqref{E:tau infty}, that $h(x)= \log(1/\tau(x))$ is a coercive function.
By~\eqref{E:dtau} and since $\gamma_\alpha(t)=t^\alpha$ for $t$ near $0$, there exists a compact subset $K\subset M$ such that
\begin{equation*}
	\rho_\alpha^2(x)\abs{\dd h_x}=\tau^{2\alpha-1}(x)\abs{\dd \tau_x}\leq c\cdot\tau^{2\alpha-1}(x),\qquad x\in M\setminus K.
\end{equation*}
Since $\tau^{2\alpha-1}\in L^\infty(M) $ for $2\alpha\geq 1$, the previous inequality and Remark~\ref{R:coinciding at infty} imply the thesis.
\end{proof}


\subsection{A Friedrich inequality for operators linear over $C^\ast$-algebras}\label{SS:Friedrich inequality}
Let $(M,g)$ be an $n$-dimensional Riemannian spin manifold with associated spinor bundle $S_M$ and Dirac operator $\D_M$.
Let $\bigl(E,\nabla^E\bigr)$ be a bundle of finitely generated projective Hilbert $A$-modules with inner product and metric connection.
Denote by $\Zop\colon\Gamma(M;S_M\tensor E)\to\Gamma(M;S_M\tensor E)$ the Dirac operator $\D_M$ twisted with the bundle $E$.
We consider the following two situations:
\begin{enumerate}
    \item $S_M$ is the complex spinor bundle $\slashed S_M$, $\bigl(E,\nabla^E\bigr)$ is a Hermitian vector bundle with metric connection and $\Zop$ is the twisted complex spin Dirac operator $\slashed \D_{M,E}$ described in Example~\ref{Ex:complex dirac};
    \item $A$ is a Real $C^\ast$-algebra, $S_M$ is the $\cl_{n,0}$-linear spinor bundle $\spinor_M$, $\bigl(E,\nabla^E\bigr)$ is a bundle of finitely generated projective Real Hilbert $A$-modules with inner product and metric connection and $\Zop$ is the twisted $\cl_{n,0}$-linear Dirac operator $\spind_{M,E}$ described in Example~\ref{Ex:real Dirac}.
\end{enumerate}
When there is no danger of confusion, we will denote the bundle $S_M$ simply by $S$.
The operator $\Zop$ is related to the scalar curvature of $g$ through the classical Lichnerowicz formula
\begin{equation}\label{E:Lichnerowicz}
    \Zop^2=\nabla^\ast\nabla+\frac{1}{4}\scal_g+\mathcal R^E,
\end{equation}
where $\nabla^\ast\nabla$ is the connection Laplacian of $S\tensor E$ and $\mathcal R^ E\colon S\tensor E\to S\tensor E$ is a bundle map depending linearly on the components of the curvature tensor $F\bigl(\nabla^ E\bigr)$ of $\nabla^ E$.
In particular, if $F\bigl(\nabla^ E\bigr)=0$ in a region $\Omega\subset M$, then $\mathcal R^E =0$ on $\Omega$.
See~\cite[\S II.8]{LM89} for more details.
The next theorem provides a slight improvement of the estimate from below of $\Zop^2$ directly following from~\eqref{E:Lichnerowicz}.


\begin{theorem}\label{T:Friedrich}
Let $(M,g)$, $\big(E,\nabla^E\big)$ and $\Zop$ be as above. 
Set
\begin{equation}\label{E:n sigma}
    \bar n\defeq \frac{n}{n-1}\qquad\text{and}\qquad
    \nscal_g(x)\defeq \frac{1}{4}\scal_g(x).
\end{equation}
Then the inequality
\begin{equation}\label{E:Friedrich}
    \left<\Zop^2 u,u\right>\geq \bar n\left<\overline\scal_gu,u\right>+\bar n\left<\cal R^E u,u\right>
\end{equation}
holds for all $u\in \Gamma_\rm c(M;S\tensor E)$.
\end{theorem}


\noindent In order to prove Theorem~\ref{T:Friedrich}, we first establish the following abstract inequality for Hilbert $C^\ast$-modules.


\begin{lemma}\label{L:Friedrich}
Let $H$ be a Hilbert module over a $C^\ast$-algebra $A$.
For $x_1,\ldots, x_N\in H$, we have
\[
    \left(\sum_{i=1}^N x_i\Biggm| \sum_{i=1}^Nx_i\right)\leq N\sum_{i=1}^N(x_i\mid x_i)
\]
where $(\cdot\mid\cdot)$ is the $A$-valued inner product of $H$.
\end{lemma}


\begin{proof}
For $x$, $y\in H$, we have
\begin{equation}\label{E:xy+yx<x^2+y^2}
    (x\mid y)+(y\mid x)\leq(x-y\mid x-y)+(x\mid y)+(y\mid x)= (x\mid x)+(y\mid y).
\end{equation}
Therefore,
\[
    \Bigg(\sum_{i=1}^N x_i\Bigg\vert \sum_{i=1}^Nx_i\Bigg)
    =\sum_{i=1}^N(x_i\mid x_i)+\sum_{i<j}\bigl\{(x_i\mid x_j)+(x_j\mid x_i)\bigr\}
    \leq N\sum_{i=1}^N(x_i\mid x_i)
\]
where the last inequality is obtained by applying Inequality~\eqref{E:xy+yx<x^2+y^2} to the terms $(x_i\mid x_j)+(x_j\mid x_i)$.
\end{proof}


\begin{proof}[Proof of Theorem~\ref{T:Friedrich}]
Let $u\in \Gamma_\rm c(M;S\tensor E)$.
Recall that the operator $\Zop$ has the local expression
\[
    \Zop u=\sum_{i=1}^n\cc(e_i)\nabla_{e_i}u
\]
where $\left\{e_1,\ldots,e_n\right\}$ is an orthonormal basis of $T_xM$, $\cc(e_i)$ is Clifford multiplication by $e_i$, and $\nabla$ is the connection on $S\tensor E$ induced by the connections on $S$ and $E$.
At a point $x\in M$, using Lemma~\ref{L:Friedrich} we obtain
\begin{multline*}
    \left<\Zop u,\Zop u\right>_x=\left<\sum_{i=1}^n\cc(e_i)\nabla_{e_i}u,\sum_{i=1}^n\cc(e_i)\nabla_{e_i}u\right>_x\\
    \leq n\sum_{i=1}^n\left<\cc(e_i)\nabla_{e_i}u,\cc(e_i)\nabla_{e_i}u\right>_x
    = n\sum_{i=1}^n\left<\nabla_{e_i}u,\nabla_{e_i}u\right>_x    
    = n\left<\nabla u,\nabla u\right>_x.
\end{multline*}
By integrating the previous inequality, we get
\begin{equation*}
    \left<\Zop^2u,u\right>=\left<\Zop u,\Zop u\right>\leq n\left<\nabla u,\nabla u\right>=n\left<\nabla^\ast\nabla u,u\right>.
\end{equation*}
Using the last inequality together with the Lichnerowicz formula~\eqref{E:Lichnerowicz}, we deduce
\[
    \left<\Zop^2 u,u\right>
    \geq \frac{1}{n}\left<\Zop^2 u,u\right>+\left<\nscal_g u,u\right>+\left<\cal R^E u,u\right>,
\]
from which Inequality~\eqref{E:Friedrich} follows.
\end{proof}


\subsection{A Lichnerowicz-type inequality for the rescaled operator}\label{SS:rescaled Lichnerowicz}
Let $(M,g)$, $S$, $E$ and $\Zop$ be as in Subsection~\ref{SS:Friedrich inequality}.
For a smooth function $\rho$, let $\Zop_\rho$ be the rescaled operator defined by~\eqref{E:rescaled operator}.
In the next proposition, we state a Lichnerowicz-type inequality for the rescaled operator.


\begin{proposition}\label{P:rescaled Lichnerowitcz}
Suppose the scalar curvature of $g$ is bounded from below by a constant $\sigma>0$.
Let $\rho\colon M\to(0,\infty)$ be a smooth function and let $\bar n$, $\bar\sigma$ be the constants defined in~\eqref{E:n sigma}.
Then the inequality
\begin{equation}
	\bigl<\Zop_\rho^2u,u\bigr> \geq 
	\frac{\bar n\omega}{1+\omega}\left<\nscal_g\rho^4u,u\right>+\frac{\bar n\omega}{1+\omega}\bigl<\cal R^E\rho^4u,u\bigr>
	-\omega\bigl<\rho^2\abs{\dd\rho}^2u,u\bigr>
\end{equation}
holds for every $\omega>0$ and every $u\in\Gamma_\rm c(M;S\tensor E)$.
\end{proposition}


\noindent The proof of this proposition is based on the following lemma.


\begin{lemma}\label{L:BK}
Let $\xi\colon M\to\RR$ be a smooth function.
Then the inequality
\begin{equation}\label{E:BL}
	\bigl<\xi\Zop(v),\xi\Zop(v)\bigr> \geq 
	\frac{\omega}{1+\omega}\bigl<\Zop(\xi v),\Zop(\xi v)\bigr>
	-\omega\bigl<\abs{\dd\xi}^2 v,v\bigr>
\end{equation}
holds for every $\omega>0$ and every $v\in\Gamma_\rm c(M;S\tensor E)$.
\end{lemma}


\begin{proof}
By direct computation,
\[
    \Zop\xi^2\Zop=\xi\Zop^2\xi-\abs{\dd\xi}^2-\Zop\xi\cc(\dd\xi)+\cc(\dd\xi)\xi\Zop.
\]
Hence,
\begin{multline}\label{E:BL11}
    \bigl<\xi\Zop(v),\xi\Zop(v)\bigr>=\bigl<\Zop(\xi v),\Zop(\xi v)\bigr>-\bigl<\abs{\dd\xi}^2 v,v\bigr>\\
    -\bigl<\cc(\dd\xi)v,\xi\Zop v\bigr>-\bigl<\xi\Zop v,\cc(\dd\xi)v\bigr>.
\end{multline}
Fix $\omega>0$ and observe that
\begin{multline*}
    0\leq\biggl<\frac{\xi}{\sqrt{\omega}}\Zop(v)-\sqrt{\omega}\cc(\dd\xi)v,\frac{\xi}{\sqrt{\omega}}\Zop(v)-\sqrt{\omega}\cc(\dd\xi)v\biggr>\\
    =\frac{1}{\omega}\bigl<\xi\Zop(v),\xi\Zop(v)\bigr>+{\omega}\bigl<\abs{\dd\xi}^2 v,v\bigr>
    -\bigl<\cc(\dd\xi)v,\xi\Zop v\bigr>-\bigl<\xi\Zop v,\cc(\dd\xi)v\bigr>.
\end{multline*}
This inequality together with~\eqref{E:BL11} yields
\[
    \bigl<\xi\Zop(v),\xi\Zop(v)\bigr>\geq \bigl<\Zop(\xi v),\Zop(\xi v)\bigr>-\frac{1}{\omega}\bigl<\xi\Zop(v),\xi\Zop(v)\bigr>-(1+\omega)\bigl<\abs{\dd\xi}^2 v,v\bigr>,
\]
which implies~\eqref{E:BL}.
\end{proof}


\begin{proof}[Proof of Proposition~\ref{P:rescaled Lichnerowitcz}]
It follows from Theorem~\ref{T:Friedrich} and Lemma~\ref{L:BK}, with $\xi=\rho$ and $v=\rho u$.
\end{proof}


\section{Generalized Gromov-Lawson operators}\label{S:GL}
In this section, we study the geometric situation when $M$ is a Riemannian spin manifold and $(E,F)$ is a pair of bundles with isomorphic typical fibers and whose supports are contained in the interior of a compact submanifold with boundary $L\subset M$.
In Subsection~\ref{SS:localized obstructions}, we define the class $\relind(M;E,F)$ as the index of a suitable elliptic differential operator $\D_{L_\rm D}^{E,F}$ over the double $L_\rm D$ of $L$.
In Subsection~\ref{SS:compatible potentials}, we use an admissible rescaling function $\rho$ and a potential $\phi$ to define a Fredholm operator $\Pop_{\rho,\phi}^{E,F}$ on $M$.
Finally, in Subsection~\ref{SS:index theorem} we show that the index of $\Pop_{\rho,\phi}^{E,F}$ coincides with $\relind(M;E,F)$.


\subsection{Localized Dirac obstructions}\label{SS:localized obstructions}
Let $(M,g)$ be an $n$-dimensional Riemannian spin manifold with associated $\ZZ_2$-graded spinor bundle $S_M=S^+_M\oplus S^-_M$.
Let $\bigl(E,\nabla^E\bigr)$ be a bundle of finitely generated projective Hilbert $A$-modules with inner product and metric connection on $M$.
Denote by $\D_{M,E}\colon\Gamma(M;S_M\tensor E)\to\Gamma(M;S_M\tensor E)$ the spin Dirac operator twisted with the bundle $E$.
Observe that $\D_{M,E}$ is odd with respect to the grading
\begin{equation}\label{E:S tensor E grading}
     S_M\tensor E= \bigl(S_M^+\tensor E\bigr)\oplus \bigl(S_M^-\tensor E\bigr).
\end{equation}
We consider the following two situations.
\begin{enumerate}[label=(\Roman*)]
    \item $M$ is even dimensional and $\bigl(E,\nabla^E\bigr)$ is a Hermitian vector bundle with metric connection.
    In this case, $S_M$ is the complex spinor bundle $\slashed S_M$ and $\D_{M,E}$ is the twisted complex spin Dirac operator $\slashed\D_{M,E}$ described in Example~\ref{Ex:complex dirac}.
    \item $A$ is a Real $C^\ast$-algebra, and $E$ is a bundle of finitely generated projective Real Hilbert $A$-modules with inner product and metric connection.
    In this case, $S_M$ is the $\cl_{n,0}$-linear spinor bundle $\spinor_M$ and $\D_{M,E}$ is the twisted $\cl_{n,0}$-linear spin Dirac operator $\spind_{M,E}$ described in Example~\ref{Ex:real Dirac}.
\end{enumerate}
When there is no danger of confusion, we will use the notation $S$ and $S^\pm$ instead of $S_M$ and $S_M^\pm$.


Let $\bigl(F,\nabla^F\bigr)$ be a second bundle of finitely generated projective Hilbert $A$-modules with inner product and metric connection over $M$.
We make the following assumption.


\begin{assumption}\label{A:GL}
\emph{The bundles have isomorphic typical fibers and are trivializable at infinity}.
This means that there exist a finitely generated projective Hilbert $A$-module $\cal V$ and a compact subset $K\subset M$ such that
\[
    \bigl(E,\nabla^E\bigr)\big|_{M\setminus K}\cong(F,\nabla^F)\big|_{M\setminus K}\cong \bigl(\underline {\cal V},\dd_{\underline {\cal V}}\bigr)\big|_{M\setminus K}
\]
where $\underline{\cal V}\to M$ denotes the trivial bundle with fiber $\cal V$ and $\dd_{\underline {\cal V}}$ denotes the trivial connection on $\underline{\cal V}$. 
In this case, we say that $K$ is an \emph{essential support} for $(E,F)$ and that $M\setminus K$ is a \emph{neighborhood of infinity}.
\end{assumption}


In this setting, we define a relative index following Gromov and Lawson~\cite{GL83}.
Let $L\subset M$ be a smooth compact submanifold with boundary, whose interior contains an essential support of $(E,F)$.
Deform the metric and the spinor bundle in such a way that they have a product structure in a tubular neighborhood of $\partial L$.
Form the double $L_\rm D\defeq L\cup_{\partial L} L^-$ of $L$, where $L^-$ denotes the manifold $L$ with opposite orientation.
The Riemannian metric $g$ induces a Riemannian metric $g_1$ on $L_\rm D$ which is symmetric with respect to $\partial L$ and has a product structure in a tubular neighborhood of $\partial L$.
The double $L_\rm D$ is a closed manifold carrying a natural spin structure induced by the spin structure of $L$.
The associated spinor bundle $S_{L_\rm D}$ has a reflection symmetry with respect to $\partial L$.
Using Assumption~\eqref{A:GL}, define $\bigl(V(E,F),\nabla^{V(E,F)}\bigr)$ as the bundle with connection on $L_\rm D$ coinciding with $\bigl(E,\nabla^E\bigr)$ over $L$ and with $\bigl(F,\nabla^F\bigr)$ over $L^-$.
Denote by $\D_{L_{\D}}^{E,F}$ the Dirac operator $\D_{L_\rm D}$ twisted with the bundle $V(E,F)$.
In the next lemma, we collect some properties of the index of the operator $\D_{L_\rm D}^{E,F}$.


\begin{lemma}\label{L:index on double}
Let $(E,F)$ and $(G,H)$ be two pairs of bundles of finitely generated projective Hilbert $A$-modules with inner product and metric connection over $M$ satisfying Assumption~\ref{A:GL}.
Let $L\subset M$ be a compact submanifold with boundary whose interior contains an essential support of both $(E,F)$ and $(G,H)$.
Then
\begin{equation}\label{E:index on double1}
    \ind\left(\D_{L_D}^{E,E}\right)=0;
\end{equation}
\begin{equation}\label{E:index on double2}
    \ind\left(\D_{L_D}^{E,F}\right)+\ind\left(\D_{L_D}^{H,G}\right)
    =\ind\left(\D_{L_D}^{E,G}\right)+\ind\left(\D_{L_D}^{H,F}\right);
\end{equation}
and
\begin{equation}\label{E:index on double3}
    \ind\left(\D_{L_D}^{E,F}\right)+\ind\left(\D_{L_D}^{F,E}\right)=0.
\end{equation}
\end{lemma}


\begin{proof}
Identity~\eqref{E:index on double1} follows from the fact that the operator $\D_{L_\rm D}^{E,E}$ is symmetric with respect to the separating hypersurface $\partial L$.
For Identity~\eqref{E:index on double2}, consider the partition $L\cup_{\partial L}L^-$ and apply Theorem~\ref{T:relative} to the operators $\D_{L_D}^{E,F}$ and $\D_{L_D}^{G,H}$.
Finally, Identity~\eqref{E:index on double3} follows from~\eqref{E:index on double1} and~\eqref{E:index on double2}.
\end{proof}


Observe that the index of $\D_{L_{\D}}^{E,F}$ does not depend on the metric.
The next proposition states that it does not depend on the choice of the submanifold $L$.


\begin{proposition}\label{P:relind}
Let $(E,\nabla^E)$ and $(F,\nabla^F)$ be a pair of bundles of finitely generated projective Hilbert $A$-modules with inner product and metric connection over $M$ satisfying Assumption~\ref{A:GL}.
Suppose $L$ and $L^\prime$ are smooth compact submanifolds with boundary of $M$ whose interiors contain an essential support of $(E,F)$.
Then the indices of $\D_{L_{\D}}^{E,F}$ and $\D_{L^\prime_{\D}}^{E,F}$ coincide. 
\end{proposition}


\begin{proof}
Observe first that it suffices to prove the thesis when one of the submanifolds is contained in the interior of the other.
To see this, consider a compact submanifold with boundary $L^{\prime\prime}\subset M$ whose interior contains both $L$ and $L^\prime$.

Using this observation, we will prove the theorem under the assumption that $L$ is contained in the interior of $L^\prime$.
Consider the Riemannian spin manifolds $(L_\rm D,g_1)$ and $(L^\prime_\rm D,g_2)$, where $g_1$ and $g_2$ are induced by $g$ as explained above.
Consider the operators $\D_{L_{\D}}^{F,F}$ on $L_\rm D$ and $\D_{L^\prime_{\D}}^{E,F}$ on $L^\prime_\rm D$.
Observe we have partitions $L_\rm D=L\cup_{\partial L}L^-$ and $L^\prime_\rm D=L\cup_{\partial L}W$, where $W=\overline{L^\prime\setminus L}\cup_{\partial L^\prime}(L^\prime)^-$.
Deform all structures to be a product in a tubular neighborhood of $\partial L$ in such a way that Assumption~\ref{A:cut-and-paste} is satisfied.
Using the cut-and-paste construction described in Subsection~\ref{SS:cut-and-paste}, we obtain the operator $\D_{L^\prime_{\D}}^{F,F}$ on $L^\prime_{\D}$ and the operator $\D_{L_\rm D}^{E,F}$ on $L_\rm D$.
By~\eqref{E:index on double1}, the indices of $\D_{L_{\D}}^{F,F}$ and $\D_{L^\prime_{\D}}^{F,F}$ vanish.
Using Theorem~\ref{T:relative}, we obtain
\begin{multline*}
    \ind\left(\D_{L^\prime_{\D}}^{E,F}\right)=\ind\left(\D_{L_{\D}}^{F,F}\right)+\ind\left(\D_{L^\prime_{\D}}^{E,F}\right)\\
    =\ind\left(\D_{L^\prime_{\D}}^{F,F}\right)+\ind\left(\D_{L_{\D}}^{E,F}\right)=\ind\left(\D_{L_{\D}}^{E,F}\right),
\end{multline*}
which concludes the proof.
\end{proof}


\noindent Proposition~\ref{P:relind} allows us to define the \emph{relative index} of the pair $(E,F)$ as the class
\begin{equation}\label{E:relind}
    \relind(M;E,F)\defeq\ind\left(\D_{L_{\D}}^{E,F}\right),
\end{equation}
where $L\subset M$ is a submanifold with boundary whose interior contains an essential support of $(E,F)$.


\begin{remark}
In the case~(I) from the beginning of this subsection, $\relind(M;E,F)\in\ZZ$.
In the case~(II), $\relind(M;E,F)\in\KO_n(A)$.
\end{remark}


This class will be used as a localized obstruction for the metric $g$ to have positive scalar curvature under some extra geometric conditions.
To this end, we will need information on the endomorphisms $\cal R^E$ and $\cal R^F$ that appear in the Lichnerowicz formula~\eqref{E:Lichnerowicz}.
We conculde this subsection presenting two examples where we can determine whether the class $\relind(M;E,F)$ vanishes and we have control on the lower bound of the endomorphisms $\cal R^E$ and $\cal R^F$.
These two examples will be used in the geometric applications of Section~\ref{S:LNP} and Section~\ref{S:BW}.


\begin{example}\label{E:GL relative}
Let $(M,g)$ be an even-dimensional Riemannian spin manifold and let $f\colon (M,g)\to(S^n,g_0)$ be a smooth map which is strictly area decreasing and locally constant at infinity.
This last condition means that there exists a compact subset $K\subset M$ such that $f$ is constant on the connected components of $M\setminus K$.
Then, using a construction of Gromov and Lawson~\cite{GL80,GL83} and estimates by Llarull~\cite{Lla98}, there exist Hermitian vector bundles with metric connections $(E,\nabla^E)$ and $(F,\nabla^F)$ satisfying Assumption~\eqref{A:GL} and such that
\begin{enumerate}[label=(\roman*)]
    \item $(F,\nabla^F)$ is the trivial bundle endowed with the trivial connection;
    \item $(E,\nabla^E)$ is pulled back from $S^n$ and satisfies
    	\[
    		{\cal R}^E_x>-\frac{n(n-1)}{4},\qquad x\in\supp(\dd f);
    	\] 
    \item the support of $\dd f$ is an essential support of $(E,F)$;
    \item if $\relind(M;E,F)$ vanishes, then $\deg(f)=0$.
\end{enumerate}
\end{example}


\begin{example}\label{E:higher relative}
Our second example makes use of higher index theory.
Let $Y$ be a closed $n$-dimensional spin manifold with fundamental group $\Gamma$.
Let $\bigl(\mathcal{L}_Y,\nabla^{\mathcal{L}_Y}\bigr)$ be the Mishchenko bundle over $Y$ endowed with the canonical flat metric connection.
Recall that $\mathcal{L}_Y$ has typical fiber $C^\ast\Gamma$.
Suppose Condition~\eqref{C:relative higher index} is satisfied.
Pick distinct points $P_1,\ldots,P_N\in Y$ and consider the open manifold $M\defeq Y\setminus\{P_1,\ldots P_N\}$.
Let $D_1,\ldots,D_N$ be pairwise disjoint $n$-dimensional disks embedded in $Y$ such that $P_j$ is in the interior of $D_j$.
Choose embedded $n$-dimensional disks $D_1^\prime,\ldots, D_N^\prime$ such that $D_j^\prime$ lies in the interior of $D_j$ and $P_j$ is in the interior of $D_j^\prime$.
Let $f\colon M\to Y$ be a smooth map collapsing each end $D_j^\prime\setminus P_j$ to the point $P_j$ and being the identity map outside of $(D_1\setminus P_1)\sqcup\ldots\sqcup (D_N\setminus P_N)$.
Let $\bigl(E,\nabla^E\bigr)$ be the flat bundle over $M$ obtained as the pullback of $\bigl(\mathcal{L}_Y,\nabla^{\mathcal{L}_Y}\bigr)$ through $f$.
Let $\bigl(F,\nabla^F\bigr)$ be the trivial bundle over $M$ with fiber $C^\ast\Gamma$, equipped with the trivial connection.
Then 
\begin{enumerate}[label=(\roman*)]
    \item $(E,\nabla^E)$ and $(F,\nabla^F)$ satisfy Assumption~\eqref{A:GL};
    \item $\relind(M;E,F)\neq 0$.
\end{enumerate}
Property~(ii) follows from Condition~\eqref{C:relative higher index} using Theorem~\ref{T:relative}.
Notice that, since the connections $\nabla^E$ and $\nabla^F$ are flat, $\cal R^E=\cal R^F=0$.
A similar construction works if we pick pairwise disjoint embedded $n$-dimensional disks $D_1,\ldots,D_N\subset Y$ and consider the open manifold $Y\setminus\bigsqcup_{j=1}^N D_j$.
\end{example}


\subsection{Compatible potentials}\label{SS:compatible potentials}
Let $(M,g)$, $\bigl(E,\nabla^E\bigr)$ and $\bigl(F,\nabla^F\bigr)$ denote the same objects as in Subsection~\ref{SS:localized obstructions}.
Suppose Assumption~\ref{A:GL} is satisfied.
Consider the twisted Dirac operators $\Qop\colon\Gamma_\rm c(M; S\tensor E)\to\Gamma_\rm c(M; S\tensor E)$ and $\Rop\colon\Gamma_\rm c(M; S\tensor F)\to\Gamma_\rm c(M; S\tensor F)$.
Recall from Subsection~\ref{SS:localized obstructions} that we have $\ZZ_2$-gradings
\begin{equation}\label{E:twisted grading}
     S\tensor E= \bigl(S^+\tensor E\bigr)\oplus \bigl(S^-\tensor E\bigr)\qquad\text{and}\qquad
     S\tensor F= \bigl(S^+\tensor F\bigr)\oplus \bigl(S^-\tensor F\bigr)
\end{equation}
and that the operators $\Qop$ and $\Rop$ are odd with respect to these gradings.
Fix an admissible rescaling function $\rho$ for $M$ and consider the rescaled operators $\Qop_\rho$ and $\Rop_\rho$ defined by~\eqref{E:rescaled operator}.
Recall from Subsection~\ref{SS:admissible rescaling} that $\Qop_\rho$ and $\Rop_\rho$ are first order formally self-adjoint elliptic differential operators.
Finally, observe that the operators $\Qop_\rho$ and $\Rop_\rho$ are odd with respect to the grading~\eqref{E:twisted grading}, i.e. they are of the form
\[
    \Qop_\rho=\begin{pmatrix}0&\Qop_\rho^-\\\Qop_\rho^+&0\end{pmatrix}\qquad\text{and}\qquad
    \Rop_\rho=\begin{pmatrix}0&\Rop_\rho^-\\\Rop_\rho^+&0\end{pmatrix}
\]
where $\Qop_\rho^+\colon\Gamma_\rm c(M; S^+\tensor E)\to \Gamma_\rm c(M; S^-\tensor E)$, 
$\Rop_\rho^+\colon\Gamma_\rm c(M; S^+\tensor F)\to \Gamma_\rm c(M; S^-\tensor F)$ and $\Qop_\rho^-$, $\Rop_\rho^-$ are formally adjoint respectively to $\Qop_\rho^+$, $\Rop_\rho^+$.

In order to construct a Fredholm operator out of the operators $\Qop_\rho$ and $\Rop_\rho$, we make use of a potential.


\begin{definition}
We say that a smooth function $\phi\colon M\to[0,\infty)$ is a \emph{compatible potential} if $\phi=0$ in a neighborhood of an essential support of $(E,F)$ and $\phi$ is constant and nonzero in a neighborhood of infinity.
\end{definition}


\noindent Fix a compatible potential $\phi$.
By Assumption~\ref{A:GL}, the bundles $E$ and $F$ are isomorphic in a neighborhood of the support of $\phi$.
Therefore, $\phi$ defines bundle maps
\[
	\phi\colon S^\pm\tensor E\longrightarrow S^\pm\tensor F\qquad\text{and}\qquad
	\phi\colon S^\pm\tensor F\longrightarrow S^\pm\tensor E.
\]
Set
\[
    W^+:=\bigl( S^+\tensor E\bigr)\oplus\bigl( S^-\tensor F\bigr)\qquad\text{and}\qquad
    W^-:=\bigl( S^+\tensor F\bigr)\oplus\bigl( S^-\tensor E\bigr).
\]
Define the operator $\Pop_{\rho,\phi}^+\colon\Gamma(W^+)\rightarrow \Gamma(W^-)$ through the formula
\[
	\Pop_{\rho,\phi}^+\defeq
		\begin{pmatrix}\phi &\Rop_\rho^- \\ \Qop_\rho^+ &-\phi \end{pmatrix}.
\]
Denote by $\Pop_{\rho,\phi}^-$ its formal adjoint and consider the graded bundle $W\defeq W^+\oplus W^-$.
The \emph{generalized Gromov-Lawson operator} associated to our data is the operator $\Pop_{\rho,\phi}^{E,F}\colon\Gamma(W)\to\Gamma(W)$ defined as
\[
	\Pop_{\rho,\phi}^{E,F}\defeq\begin{pmatrix}0 & \Pop_{\rho,\phi}^- \\ \Pop_{\rho,\phi}^+ & 0\end{pmatrix}.
\]
By construction, $\Pop_{\rho,\phi}^{E,F}$ is an odd formally self-adjoint elliptic differential operator of order one.
When there is no danger of confusion, we will denote $\Pop_{\rho,\phi}^{E,F}$ simply by $\Pop_{\rho,\phi}$.

\begin{theorem}\label{T:Fredholm}
For every admissible rescaling function $\rho$ and every compatible potential $\phi$, the pair $(M,\Pop_{\rho,\phi})$ is complete and the operator $\Pop_{\rho,\phi}^2$ is uniformly positive at infinity.
\end{theorem}


\noindent The proof of this theorem is based on the following two lemmas.


\begin{lemma}\label{L:Hilbert inequality}
Let $U$, $V$ be Hilbert $A$-modules and let $T\colon U\to V$ be an adjointable operator such that $T^\ast T=b^2\id_U$, for some constant $b>0$.
Then for every $\eta\in U$ and $\theta\in V$ we have
\begin{equation}\label{E:Hilbert inequality}
    \left( T\eta\mid \theta \right)_V+ \left( \theta\mid T\eta \right)_V \geq - \left(b\eta\mid \eta\right)_U-\left(b\theta\mid \theta\right)_V,
\end{equation}
where $(\cdot\mid\cdot)_U$ and $(\cdot\mid\cdot)_V$ are the $A$-valued inner products respectively of $U$ and $V$.
\end{lemma}


\begin{proof}
Pick $\eta\in U$ and $\theta\in V$.
We have
\begin{multline*}
    0\leq\left(b^{-1/2}T\eta+b^{1/2}\theta \Bigm| b^{-1/2}T\eta+b^{1/2}\theta\right)_V   \\
    =\left(b^{-1}T^\ast T\eta\mid \eta\right)_U+\left(b\theta\mid \theta\right)_V
    +\left(T\eta\mid \theta\right)_V+\left(\theta\mid T\eta\right)_V\\
    =\left(b\eta\mid\eta\right)_U+\left(b\theta\mid \theta\right)_V
    +\left(T\eta\mid \theta\right)_V+\left(\theta\mid T\eta\right)_V,
\end{multline*}
from which~\eqref{E:Hilbert inequality} follows.
\end{proof}


\begin{lemma}\label{L:Fredholm}
Let $w\in\Gamma_\rm c(M;W^+)$, $u\in\Gamma_\rm c(M;S^+\tensor E)$, and $v\in\Gamma_\rm c(M;S^-\tensor F)$ be such that $w=u\oplus v$. Then
\begin{multline}\label{E:Fred4}
\left<\Pop_{\rho,\phi}^+w, \Pop_{\rho,\phi}^+w\right>\geq
	\left<\Qop_\rho^+u,  \Qop_\rho^+u\right>
	+\left<\phi u,  \phi u\right>
	-\left<\rho^2 \abs{\dd\phi}u,u\right>\\
	+\left<\Rop_\rho^-v,  \Rop_\rho^-v\right>
	+\left<\phi v,  \phi v\right>
	-\left<\rho^2 \abs{\dd\phi}v,v\right>.
\end{multline}
Analogously, 
\begin{multline}\label{E:Fred3}
\left<\Pop_{\rho,\phi}^-\bar w, \Pop_{\rho,\phi}^-\bar w\right>\geq
	\left<\Rop_\rho^+\bar v,  \Rop_\rho^+\bar v\right>
	+\left<\phi\bar v,  \phi\bar v\right>
	-\left<\rho^2 \abs{\dd\phi}\bar v,\bar v\right>
    \\
	+\left<\Qop_\rho^-\bar u,  \Qop_\rho^-\bar u\right>
	+\left<\phi\bar u,  \phi\bar u\right>
	-\left<\rho^2 \abs{\dd\phi}\bar u,\bar u\right>
\end{multline}
for every $\bar w\in\Gamma_\rm c(M;W^-)$, $\bar u\in\Gamma_\rm c(M;S^-\tensor E)$, and $\bar v\in\Gamma_\rm c(M; S^+\tensor F)$ such that $\bar w=\bar v\oplus\bar u$.
\end{lemma}


\begin{proof}
We have
\begin{equation*}
	\Pop_{\rho,\phi}^+w=
	\left(\phi u+\Rop_\rho^-v\right) \oplus \left(\Qop_\rho^+u-\phi v\right),
\end{equation*}
from which
\begin{multline}\label{E:Fred0}
\left<\Pop_{\rho,\phi}^+w, \Pop_{\rho,\phi}^+w\right> 
	=\left<\Qop_\rho^+u,  \Qop_\rho^+u\right>
	+\left<\phi u,  \phi u\right>
	+\left<\bigl(\Rop_\rho^+\phi-\phi\Qop_\rho^+\bigr)u,  v\right>\\
	+\left<\Rop_\rho^-v,  \Rop_\rho^-v\right>
	+\left<\phi v,  \phi v\right>
	+\left<v, \bigl(\Rop_\rho^+\phi-\phi\Qop_\rho^+\bigr)u\right>.
\end{multline}
Let us now analyze the operator $\bigl(\Rop_\rho^+\phi-\phi\Qop_\rho^+\bigr)\colon\Gamma(M;S^+\tensor E)\to\Gamma(M;S^-\tensor F)$.
By Assumption~\ref{A:GL}, we have isomorphisms
\begin{equation}\label{E:iso at infty}
    S^+\tensor E\cong S^+\tensor F\qquad S^-\tensor E\cong S^-\tensor F\qquad\textnormal{ on } M\setminus K.
\end{equation}
Since $\phi$ vanishes in a neighborhood of $K$, using the isomorphisms~\eqref{E:iso at infty} Clifford multiplication by $\dd\phi$ defines a bundle map $\tilde\cc(\dd\phi)\colon S^+\tensor E\to S^-\tensor F$.
When $\dd\phi_x=0$, $\tilde\cc(\dd\phi_x)$ is the zero map.
When $\dd\phi_x\neq 0$, $\tilde\cc(\dd\phi_x)^\ast\tilde\cc(\dd\phi_x)=-\abs{\dd\phi_x}^2$.
Observe that, under the isomorphisms~\eqref{E:iso at infty}, the operators $\Qop^+_\rho$ and $\Qop^-_\rho$ correspond respectively to the operators $\Rop^+_\rho$ and $\Rop^-_\rho$.
Therefore,
\[
    \Rop_\rho^+\phi-\phi\Qop_\rho^+=\rho^2\tilde\cc(\dd \phi).
\]
Moreover, we have 
\begin{multline}\label{E:Fredholm1}
    \bigl<\rho^2(x)\tilde\cc(\dd\phi_x)u(x)\mid v(x)\bigr>_x+\bigl<v(x)\mid\rho^2(x)\tilde\cc(\dd\phi_x)u(x)\bigr>_x\\
    \geq
    -\bigl<\rho^2(x)\abs{\dd\phi_x}u(x)\mid u(x)\bigr>_x-\bigl<\rho^2(x)\abs{\dd\phi_x}v(x)\mid v(x)\bigr>_x
\end{multline}
for all $x\in M$.
When $\dd\phi_x=0$, this inequality is trivial.
When $\dd\phi_x\neq 0$, it follows from Lemma~\ref{L:Hilbert inequality} by setting $U=S^+_x\tensor E_x$, $V=S^-_x\tensor F_x$, $\eta=u(x)$, $\theta=v(x)$, $T=\rho^2(x)\cc(\dd\phi_x)$ and $b=\rho^2(x)\abs{\dd\phi_x}$.
Using~\eqref{E:Fredholm1}, we obtain
\begin{multline*}
    \left<\bigl(\Rop_\rho^+\phi-\phi\Qop_\rho^+\bigr)u,  v\right>
    +\left<v, \bigl(\Rop_\rho^+\phi-\phi\Qop_\rho^+\bigr)u\right>\\
    =\int_{M}\Bigl\{\bigl<\rho^2(x)\tilde\cc(\dd\phi_x)u(x)\mid v(x)\bigr>_x
    +\bigl<v(x)\mid\rho^2(x)\tilde\cc(\dd\phi_x)u(x)\bigr>_x\Bigr\}\,\dd\mu_g(x)\\
    \geq-\int_{M}\Bigl\{\bigl<\rho^2(x)\abs{\dd\phi_x}u(x)\mid u(x)\bigr>_x
    +\bigl<\rho^2(x)\abs{\dd\phi_x}v(x)\mid v(x)\bigr>_x\Bigr\}\,\dd\mu_g(x)\\
    =-\left<\rho^2\abs{\dd\phi}u,u\right>-\left<\rho^2 \abs{\dd\phi}v,v\right>.
\end{multline*}
Finally, Inequality~\eqref{E:Fred4} follows from this last inequality and~\eqref{E:Fred0}.
Inequality~\eqref{E:Fred3} is proved in a similar way.
\end{proof}


\begin{proof}[Proof of Theorem~\ref{T:Fredholm}]
The completeness of the pair $(M,\Pop_{\rho,\phi})$ follows from Proposition~\ref{P:adm->complete} and Remark~\ref{R:completeness depends on symbol}.
Moreover, since $\rho\leq1$, from Lemma~\ref{L:Fredholm} we deduce
\begin{equation*}\label{E:Pop^2 Fredholm estimate}
        \left<\Pop_{\rho,\phi}^2w,w\right>\geq \left<\left(\phi^2-\abs{\dd\phi}\right)w,w\right>,\qquad w\in\Gamma_c(M;W).
\end{equation*}
Since $\phi$ is a compatible potential, the previous inequality implies that $\Pop_{\rho,\phi}^2$ is uniformly positive at infinity.
\end{proof}


From Theorem~\ref{T:Fredholm} and the results of Subsection~\ref{SS:relative completeness}, the class $\ind\left(\Pop_{\rho,\phi}\right)$ is well defined, for every admissible rescaling function $\rho$ and every compatible potential $\phi$.
In the case~(I) from Subsection~\ref{SS:localized obstructions}, $\ind\left(\Pop_{\rho,\phi}\right)\in\ZZ$.
In the case~(II), $\ind\left(\Pop_{\rho,\phi}\right)\in\KO_n(A)$.


\subsection{The index theorem}\label{SS:index theorem}
Let $(M,g)$, $\bigl(E,\nabla^E\bigr)$ and $\bigl(F,\nabla^F\bigr)$ denote the same objects as in Subsection~\ref{SS:localized obstructions}.
Suppose Assumption~\ref{A:GL} is satisfied.


\begin{theorem}\label{T:index}
For every admissible rescaling function $\rho$ and every compatible potential $\phi$, the classes $\ind \left(\Pop_{\rho,\phi}\right)$ and $\relind(M;E,F)$ coincide.
\end{theorem}

In order to prove this theorem, we first establish some stability properties of the index of $\Pop_{\rho,\phi}$.


\begin{lemma}\label{L:stability}
Let $\rho$ be an admissible rescaling function and let $\phi$ be a compatible potential. 
Then,
\begin{enumerate}[label=\emph{(\alph*)}]
    \item if $\rho^\prime$ is a second admissible rescaling function coinciding with $\rho$ in a neighborhood of infinity, then $\ind\left(\Pop_{\rho^\prime,\phi}\right)=\ind\left(\Pop_{\rho,\phi}\right)$;
    \item if $\phi^\prime$ is a second compatible potential, then $\ind\left(\Pop_{\rho,\phi^\prime}\right)=\ind\left(\Pop_{\rho,\phi}\right)$;
    \item if $E=F$, then $\ind\left(\Pop_{\rho,\phi}\right)=0$.
\end{enumerate}
\end{lemma}


\begin{proof}
By Remark~\ref{R:coinciding at infty}, the function $\rho_t\defeq t\rho^\prime+(1-t)\rho$ is an admissible rescaling function, for $t\in[0,1]$.
Part~(a) follows by observing that $\Bigl\{\Pop_{\rho_t,\phi}\left(\Pop_{\rho_t,\phi}^2+1\right)^{-1/2}\Bigr\}$, with $0\leq t\leq 1$, is a continuous path of Fredholm operators.

Observe now that the function $\phi_t\defeq t\phi^\prime+(1-t)\phi$ is a compatible potential, for $t\in[0,1]$.
Part~(b) follows by observing that $\Bigl\{\Pop_{\rho,\phi_t}\left(\Pop_{\rho,\phi_t}^2+1\right)^{-1/2}\Bigr\}$, with $0\leq t\leq 1$, is a continuous path of Fredholm operators.

Finally, suppose that $E=F$.
In this case, the operator $\Pop_{\rho,1}$ is well defined and Fredholm.
Moreover, by the computations of Lemma~\ref{L:Fredholm}, the operator $\Pop_{\rho,1}^2$ is uniformly positive and $\ind\left(\Pop_{\rho,1}\right)=0$.
By considering the functions $\phi_t=t+(1-t)\phi$ and arguing as in Part~(b), we deduce Part~(c).
\end{proof}


We will now establish an additivity formula for generalized Gromov-Lawson operators, from which we will deduce Theorem~\ref{T:index}.
Let us consider the following situation.
Let $(E,F)$ and $(G,H)$ be two pairs over $(M,g)$ satisfying Assumption~\ref{A:GL}.
Let $K\subset M$ be an essential support of both $(E,F)$ and $(G,H)$ and let $L\subset M$ be a compact submanifold with boundary whose interior contains $K$.
Let $\phi$ be a compatible potential vanishing in a neighborhood of $L$ and let $\rho$ be an admissible rescaling function such that $\rho=1$ in a neighborhood of $L$.
Denote respectively by $\Pop_{\rho,\phi}^{E,F}$ and $\Pop_{\rho,\phi}^{G,H}$ the generalized Gromov-Lawson operators associated to these data.


\begin{lemma}\label{L:GL additivity}
In the above situation, we have
\begin{equation}\label{E:GL additivity2}
    \ind\left(\Pop_{\rho,\phi}^{E,F}\right)+\ind\left(\D_{L_D}^{G,H}\right)=
    \ind\left(\Pop_{\rho,\phi}^{G,H}\right)+\ind\left(\D_{L_D}^{E,F}\right).
\end{equation}
\end{lemma}


\begin{proof}
Let $S_{L_\rm D}$ and $\D_{L_\rm D}$ be the $\ZZ_2$-graded spinor bundle and odd Dirac operator associated to the spin manifold $L_\rm D$.
Consider the $\ZZ_2$-graded bunlde $V=V^+\oplus V^-$ over $L_\rm D$, where
\[
    V^+:=\left( S_{L_\rm D}^+\tensor V(G,E)\right)\oplus\left( S_{L_\rm D}^-\tensor V(H,F)\right)
\]
\[
    V^-:=\left( S_{L_\rm D}^+\tensor V(H,F)\right)\oplus\left( S_{L_\rm D}^-\tensor V(G,E)\right)
\]
and where $V(G,E)$ and $V(H,F)$ are the bundles defined in~Subsection~\ref{SS:localized obstructions}.
Define the operator $\Top^+\colon\Gamma(L_\rm D;V^+)\rightarrow \Gamma(L_\rm D;V^-)$ through the formula
\[
	\Top^+\defeq
		\begin{pmatrix}0 &\Bigl(\D_{L_\rm D}^{H,F}\Bigr)^- \\ \Bigl(\D_{L_\rm D}^{G,E}\Bigr)^+ & 0 \end{pmatrix}.
\]
Let $\Pop_{H,F}^{G,E}\colon\Gamma(V)\to\Gamma(V)$ be the operator defined as
\[
	\Pop_{H,F}^{G,E}\defeq\begin{pmatrix}0 & \Top^- \\ \Top^+ & 0\end{pmatrix},
\]
where $\Top^-$ is the formal adjoint of $\Top^+$.
Observe that $\Pop_{H,F}^{G,E}$ is an odd formally self-adjoint elliptic differential operator of order one.
By construction,
\[
    \ind\left(\Pop_{H,F}^{G,E}\right)=\ind\left(\D_{L_D}^{G,E}\right)-\ind\left(\D_{L_D}^{H,F}\right).
\]
Using~\eqref{E:index on double2} and~\eqref{E:index on double3}, we obtain
\begin{equation}\label{E:GL additivity3}
    \ind\left(\Pop_{H,F}^{G,E}\right)
    =\ind\left(\D_{L_D}^{G,H}\right)-\ind\left(\D_{L_D}^{E,F}\right).
\end{equation}
From~\eqref{E:index on double2} and~\eqref{E:GL additivity3}, we have
\begin{equation}\label{E:GL additivity4}
    \ind\left(\Pop_{F,F}^{E,E}\right)=0.
\end{equation}
We now relate the operator $\Pop_{H,F}^{G,E}$ to the operator $\Pop^{E,F}_{\rho,\phi}$ on $M$.
Consider the partitions $M=L\cup_{\partial L} \left(M\setminus L\right)$ and $L_\rm D=L\cup_{\partial L}L^-$.
Modify the Riemannian metrics and the Clifford structures on $M$ and $L_\rm D$ in a tubular neighborhood of $\partial L$ in such a way that Assumption~\ref{A:cut-and-paste} is satisfied.
Using the cut-and-paste construction described in Subsection~\ref{SS:cut-and-paste}, we obtain the operator $\Pop^{E,F}_{\rho,\phi}$ on $M$ and the operator $\Pop_{H,F}^{G,E}$ on $L_\rm D$.
From Identities~\eqref{E:GL additivity3},~\eqref{E:GL additivity4} and Theorem~\ref{T:relative}, we deduce
\begin{multline*}
    \ind\left(\Pop^{E,F}_{\rho,\phi}\right)+\ind\left(\D_{L_D}^{G,H}\right)-\ind\left(\D_{L_D}^{E,F}\right)\\
    =\ind\left(\Pop^{E,F}_{\rho,\phi}\right)+\ind\left(\Pop_{H,F}^{G,E}\right)\\
    =\ind\left(\Pop^{G,H}_{\rho,\phi}\right)+\ind\left(\Pop_{F,F}^{E,E}\right)=\ind\left(\Pop^{G,H}_{\rho,\phi}\right),
\end{multline*}
which implies Identity~\eqref{E:GL additivity2}.
\end{proof}


\begin{proof}[Proof of Theorem~\ref{T:index}]
Let $\rho$ be an admissible rescaling function, let $\phi$ be a compatible potential, and let $L\subset M$ be a compact submanifold with boundary whose interior contains an essential support of $(E,F)$.
Using Part~(a) and Part~(b) of Lemma~\ref{L:stability}, we assume that, in a neighborhood of $L$, we have $\phi=0$ and $\rho=1$.
Let $\bigl(\underline {\cal V},\dd_{\underline {\cal V}}\bigr)$ be as in Assumption~\ref{A:GL}.
By Part~(c) of Lemma~\ref{L:stability} and Identity~\eqref{E:index on double1}, the indices of $\Pop^{\underline {\cal V},\underline {\cal V}}_{\rho,\phi}$ and $\D_{L_{\D}}^{\underline {\cal V},\underline {\cal V}}$ vanish.
Using Lemma~\ref{L:GL additivity}, we deduce
\begin{multline*}
    \ind\left(\Pop^{E,F}_{\rho,\phi}\right)=
    \ind\left(\Pop^{E,F}_{\rho,\phi}\right)+\ind\left(\D_{L_{\D}}^{\underline {\cal V},\underline {\cal V}}\right)\\
    =\ind\left(\Pop^{\underline {\cal V},\underline {\cal V}}_{\rho,\phi}\right)+\ind\left(\D_{L_{\D}}^{E,F}\right)
    =\ind\left(\D_{L_{\D}}^{E,F}\right),
\end{multline*}
which concludes the proof.
\end{proof}


\section{The long neck problem}\label{S:LNP}
This section is devoted to proving a long neck principle in the spin setting.
Suppose $(X,g)$ is a compact Riemannian spin manifold with boundary and let $(E,F)$ be a pair of bundles with isomorphic typical fibers and whose supports are contained in the interior of $X$.
We will give conditions on the lower bound of $\scal_g$ and the distance between the supports of $E$, $F$ and $\partial X$ so that $\relind(X^\rm o;E,F)$ must vanish, where $X^\rm o$ is the interior of $X$.
To this end, we will use the distance function from $\partial X$ to construct a generalized Gromov-Lawson operator $\Pop_{\rho,\phi}^{E,F}$ on $X^\rm o$.
In Subsection~\ref{SS:ALNP}, we prove a vanishing theorem for the operator $\Pop_{\rho,\phi}^{E,F}$, from which we deduce an abstract long neck principle.
As applications, we prove Theorem~\ref{T:LNP} and Theorem~\ref{T:HLN} respectively in Subsections~\ref{SS:LNP} and~\ref{SS:HLN}.
Finally, in Subsection~\ref{SS:proof of Theorem C} we use a generalized Gromov-Lawson operator on a complete manifold to prove Theorem~\ref{T:complete metrics on torus}.


\subsection{A vanishing theorem on compact manifolds with boundary}\label{SS:ALNP}
We consider the following setup.
Let $(X,g)$ be a compact $n$-dimensional Riemannian spin manifold with boundary $\partial X$.
By removing the boundary, we obtain the open manifold $X^\rm o\defeq X\setminus \partial X$.
The metric $g$ induces an incomplete metric on $X^\rm o$, that we denote by the same symbol.
Let $(E,\nabla^E)$ and $(F,\nabla^F)$ be bundles of finitely generated projective Hilbert $A$-modules with inner product and metric connection over $(X^\rm o,g)$.
Suppose Assumption~\ref{A:GL} is satisfied.


\begin{definition}
For an essential support $K$ of $(E,F)$, a \emph{$K$-bounding function} is a smooth function $\nu\colon X^\rm o\to [0,\infty)$ such that $\nu=0$ on $X^\rm o\setminus K$ and
\begin{equation}\label{C:R^E>-nu}
    {\cal R}^E_x\geq -\nu(x)\qquad\text{and}\qquad {\cal R}^F_x\geq -\nu(x)\qquad x\in K.
\end{equation}
We say that $\nu$ is a \emph{bounding function} if it is a $K$-bounding function for some essential support $K$ of $(E,F)$.
\end{definition}


The next theorem states an abstract ``long neck principle'' for compact Riemannian spin manifolds with boundary.


\begin{theorem}\label{T:GLN}
Let $K$ be an essential support of $(E,F)$, let $\nu$ be a $K$-bounding function, and let $\sigma$ be a positive constant.
Suppose that 
\begin{equation}\label{C:sigma>nu}
    \frac{\scal_g(x)}{4}>\nu(x),\qquad x\in K;
\end{equation}
\begin{equation}\label{C:scal>sigma}
    {\scal_g(x)}\geq\sigma,\qquad x\in X^\rm o\setminus K;
\end{equation}
and
\begin{equation}\label{C:dist}
    \dist\bigl(K,\partial X\bigr)>\pi\sqrt{\frac{n-1}{n\sigma}}.
\end{equation}
Then $\relind(X^\rm o;E,F)=0$.
\end{theorem}


In order to prove this theorem, we will make use of the index theory developed in Section~\ref{S:GL}.
For an admissible rescaling function $\rho$ and a compatible potential $\phi$ on $X^\rm o$, let $\Pop_{\rho,\phi}\colon\Gamma_\rm c(X^\rm o;W)\to\Gamma_\rm c(X^\rm o;W)$ be the associated generalized Gromov-Lawson operator.
For the definition of the bundle $W$ and the operator $P_{\rho,\phi}$ and their properties, see Subsection~\ref{SS:compatible potentials}.
We start with proving an estimate for the operator $\Pop_{\rho,\phi}^2$ in this setting.
As in Section~\ref{S:admissible rescaling}, we use the notation $\bar n={n/(n-1)}$ and $\nscal_g(x)=\scal_g(x)/4$.


\begin{lemma}\label{L:LWS}
Let $\rho$ be an admissible rescaling function, let $\phi$ be a compatible potential and let $\nu$ a bounding function.
Then, for every $w\in\Gamma_\rm c(M;W)$ and every $\omega>0$, we have
\[
    \bigl<\Pop^2_{\rho,\phi}w,w\bigr>\geq\bigl<\Phi_{\rho,\phi}^{\omega,\nu}w,w\bigr>,
\]
where $\Phi_{\rho,\phi}^{\omega,\nu}\colon M\to\RR$ is the smooth function defined by the formula
\begin{multline}\label{E:Phi}
    \Phi_{\rho,\phi}^{\omega,\nu}(x)\defeq 
    \frac{\bar n\omega}{1+\omega}\rho^4(x)\nscal_g(x)
    -\frac{\bar n\omega}{1+\omega}\rho^4(x)\nu(x)\\
    -\omega\rho^2(x)\abs{\dd\rho_x}^2+\phi^2(x)-\rho^2(x) \abs{\dd\phi_x}
\end{multline}
for $x\in X^\rm o$.
\end{lemma}


\begin{proof}
It follows from Lemma~\ref{L:Fredholm} and Proposition~\ref{P:rescaled Lichnerowitcz}.
\end{proof}


\begin{corollary}\label{C:lnp abstract vanishing}
Suppose there exist an admissible rescaling function $\rho$, a compatible potential $\phi$, a bounding function $\nu$, and positive constants $\omega$ and $c$ such that
\begin{equation}\label{E:Phi>c}
    \Phi_{\rho,\phi}^{\omega,\nu}(x)\geq c,\qquad x\in X^\rm o.
\end{equation}
Then the class $\relind(X^\rm o;E,F)$ vanishes.
\end{corollary}


\begin{proof}
From Lemma~\ref{L:LWS}, Condition~\eqref{E:Phi>c} implies that the operator $\Pop_{\rho,\phi}^2$ is invertible.
Now the thesis follows from Theorem~\ref{T:index}.
\end{proof}


For the remaininig part of this section, we use the notation
\[
    \bar\sigma\defeq\frac{\sigma}{4}
\]
for a positive constant $\sigma$.

\begin{lemma}\label{L:potential}
Suppose $\Lambda>\pi/\sqrt{\bar n\sigma}$ for some constant $\sigma>0$.
Then there exist constants $\omega>0$, $L\in(0,\Lambda)$ and a smooth function $Y\colon(0,\infty)\to [0,\infty)$ satisfying
\begin{enumerate}[label=\emph{(\arabic*)}]
   \item $Y=0$ in a neighborhood of $[\Lambda,\infty)$;
   \item $Y=Y_0$ in a neighborhood of $(0,L]$, where $Y_0$ is a constant such that $Y^2_0>\omega/4$;
   \item there exists a constant $c>0$ such that
   \begin{equation*}
       \frac{\bar n\bar\sigma\omega}{1+ \omega}L^2+Y^2(t)-L\abs{Y^\prime(t)}\geq c,
   \end{equation*}
   for $t$ varying in a neighborhood of $[L,\Lambda]$.
\end{enumerate}
Moreover, $L$ can be chosen arbitrary small.
\end{lemma}


\begin{proof}
Write $\sigma=\sigma_1+\sigma_2$, with $\sigma_1$ and $\sigma_2$ positive numbers.
We use the notation $\bar\sigma_1\defeq\sigma_1/4$ and $\bar\sigma_2\defeq\sigma_2/4$.
Observe that $\frac{1+\omega}{\omega}\to 1$, as $\omega\to\infty$.
Since $\Lambda>\pi/\sqrt{\bar n\sigma}$, we choose $\omega>0$ and $\sigma_1\in(0,\sigma)$ such that
\begin{equation}\label{E:sigma_1}
    \Lambda>\pi\sqrt{\frac{1+\omega}{\bar n\sigma_1\omega}}.
\end{equation}
This is achieved by taking $\omega$ large enough and $\sigma_1$ close enough to $\sigma$.

For positive constants $A$ and $B$, consider the function
\[
    Y_{A,B}(t)=AB\tan\bigl(A(\Lambda-t)\bigr),\qquad \Lambda-\frac{\pi}{2A}<t<\Lambda+\frac{\pi}{2A}.
\]
Observe that $Y_{A,B}$ satisfies
\begin{equation}\label{E:YAB}
	A^2B^2+Y_{A,B}^2(t)-B\abs{Y_{A,B}^\prime(t)}=0\qquad\text{and}\qquad Y_{A,B}(\Lambda)=0.
\end{equation}
Moreover,
\begin{equation}\label{E:YAB1}
	Y_{A,B}(t)\to\infty,\qquad\text{as }t\to\left(\Lambda-\frac{\pi}{2A}\right)^+.
\end{equation}
Observe that the point at which the function $Y_{A,B}$ goes to infinity is uniquely determined by the constant $A$.

We now make the following choice for $A$ and $B$.
Set
\[
    A\defeq\sqrt{\frac{\bar n\bar\sigma_1\omega}{1+\omega}}.
\]
Notice that
\[
    \frac{\pi}{2A}=\frac{\pi}{2}{\sqrt{\frac{1+\omega}{\bar n\bar\sigma_1\omega}}}=\pi\sqrt{\frac{1+\omega}{\bar n\sigma_1\omega}}.
\]
Using~\eqref{E:sigma_1}, we deduce 
\[
    \Lambda-\frac{\pi}{2A}=\Lambda-\pi\sqrt{\frac{1+\omega}{\bar n\sigma_1\omega}}>0.
\]
Choose
\[
    L\in\left(0,\Lambda-\pi\sqrt{\frac{1+\omega}{\bar n\sigma_1\omega}}\right)
\]
and set $B\defeq L$.
With this choice for $A$ and $B$, we obtain the function
\[
   \tilde Y(t)\defeq \sqrt{\frac{\bar n\bar\sigma_1\omega}{1+\omega}}L\tan\left(\sqrt{\frac{\bar n\bar\sigma_1\omega}{1+\omega}}(\Lambda-t)\right),
   \qquad \Lambda-\pi\sqrt{\frac{1+\omega}{\bar n\sigma_1\omega}}<t<\Lambda+\pi\sqrt{\frac{1+\omega}{\bar n\sigma_1\omega}}.
\]
From~\eqref{E:YAB}, we deduce that 
\begin{equation}\label{C:Y_eta(Lambda)=0}
	\tilde Y(\Lambda)=0
\end{equation} 
and
\begin{equation}\label{C:inequality for Y_eta}
    \frac{\bar n\bar\sigma\omega}{1+\omega}L^2+\tilde Y^2(t)-L\abs{\tilde Y^\prime(t)}
    =\frac{\bar n\bar\sigma_2 \omega}{1+\omega}\eta^2\Lambda^2>0.
\end{equation}
Moreover, by~\eqref{E:YAB1} we deduce that we can choose $T\in\left(\Lambda-\pi\sqrt{\frac{1+\omega}{\bar n\sigma_1\omega}},\Lambda\right)$ such that
\[
    \tilde Y^2(T)>\frac{\omega}{4}.
\]

For $\epsilon>0$ small enough, let $\tilde Y_\epsilon\colon (0,\infty)\to [0,\infty)$ be a smooth function such that 
\begin{itemize}
    \item $\tilde Y_{\epsilon}$ is constant in a neighborhood of $(0,T]$;
    \item $\tilde Y_{\epsilon}=0$ in a neighborhood of $[\Lambda,\infty)$;
    \item $\tilde Y_{\epsilon}$ is $\epsilon$-close to $\tilde Y$ on $[T,\Lambda]$;
    \item $\abs{\tilde Y_{\epsilon}^\prime(t)}\leq\abs{\tilde Y^\prime(t)}$, for $t\in(T,\Lambda)$.
\end{itemize}
Since $L<T$, for $\epsilon$ small enough, we obtain a function satisfying Properties~(1),~(2) and~(3).
The final statement follows from the fact that $L$ can be arbitrarily chosen in the set $\left(0,\Lambda-\pi\sqrt{\frac{1+\omega}{\bar n\sigma_1\omega}}\right)$, where $\sigma_1$ and $\omega$ satisfy~\eqref{E:YAB}.
\end{proof}


\begin{lemma}\label{L:diff inequality}
Suppose $\Lambda>\pi/\sqrt{\bar n\sigma}$.
Then there exist a constant $\omega >0$ and smooth functions $Z\colon (0,\infty)\to (0,1]$ and $Y\colon(0,\infty)\to [0,\infty)$ satisfying
\begin{enumerate}[label=\emph{(\alph*)}]
    \item $Z(t)=d\sqrt{t}$ near $0$, for some constant $d>0$;
    \item $Z$ is constant in a neighborhood of $[\Lambda,\infty)$;
    \item $Y$ is constant near $0$;
    \item $Y=0$ in a neighborhood of $[\Lambda,\infty)$;
    \item there exists a constant $c>0$ such that the functions $Z$ and $Y$ satisfy the differential inequality
    \begin{equation*}
            \qquad \frac{\bar n\bar\sigma \omega }{1+\omega }Z^4(t)
            -\omega  Z^2(t)\abs{Z^\prime(t)}^2+Y^2(t)-Z^2(t) \abs{Y^\prime(t)}\geq c,
    \end{equation*}
    for all $t>0$.
\end{enumerate}
\end{lemma}


\begin{proof}
Pick constants $\omega >0$, $L\in(0,\Lambda]$, with $L\leq 1$, and a smooth function $Y\colon(0,\infty)\to [0,\infty)$ satisfying Properties~(1)--(3) of Lemma~\ref {L:potential}.
Observe that Properties~(1) and~(2) imply that $Y$ satisfies Properties~(c) and~(d).
Using Property~(2) of Lemma~\ref {L:potential}, pick a smooth function $Z\colon (0,\infty)\to (0,1]$ and a constant $\delta>0$ such that
\begin{enumerate}[label=(\roman*)]
	\item $Z=\sqrt{L}$ in a neighborhood of $[L,\infty)$;
	\item $Z(t)=(1+\delta)\sqrt t$ for $t$ near $0$ and $Z(t)\leq(1+\delta)\sqrt t$ in a neighborhood of $(0,L]$;
	\item $0\leq Z^\prime(t)\leq (1+\delta)/(2\sqrt t)$ for all $t>0$;
	\item $Y_0^2>(1+\delta)^4\omega/4$.
\end{enumerate}
By Properties~(i) and (ii), it follows that $Z$ satisfies~(a) and~(b).
In order to conclude the proof, it remains to show that $Y$ and $Z$ satisfy ~(e).

Let $U_0$ be an open neighborhood of $(0,L]$ such that $Y=Y_0$ on $U_0$.
Let $U_1$ be an open neighborhood of $[L,\infty)$ such that $Z=\sqrt{L }$ on $U_1$.
We will prove Property~(e) by analyzing separately these two open sets.
Let us begin with $U_0$.
On this set, $Y=Y_0$.
Moreover, by Properties~(ii) and~(iii) we have 
\[
	Z^2(t)\abs{Z^\prime(t)}^2 \leq \frac{(1+\delta)^4}{4},\qquad t\in U_0.
\]
Therefore, using Property~(iv) we obtain
\begin{equation}\label{E:U_1}
	\frac{\bar n\bar\sigma \omega}{1+\omega}Z^4(t)
	-\omega Z^2(t)\abs{Z^\prime(t)}^2+Y^2(t)-Z^2(t) \abs{Y^\prime(t)}
	\geq
	Y_0^2-(1+\delta)^4 \frac{\omega}{4}>0,
\end{equation}
for every $t\in U_0 $.
Let us now analyze $U_1$.
On this set, $Z=\sqrt{L}$.
Therefore, using Property~(3) of Lemma~\ref {L:potential}, we deduce that there exists a constant $c>0$ such that
\begin{multline} \label{E:U_2}
	\frac{\bar n\bar\sigma \omega}{1+\omega}Z^4(t)
	-\omega Z^2(t)\abs{Z^\prime(t)}^2+Y^2(t)-Z^2(t) \abs{Y^\prime(t)}\\
	=	\frac{\bar n\bar\sigma \omega}{1+\omega} L^2
	+Y^2(t)-L \abs{Y^\prime(t)}\geq c,
\end{multline}
for every $t\in U_1 $.
Since $U_0\cup U_1=(0,\infty)$,~\eqref {E:U_1} and~\eqref {E:U_2} imply that $Y$ and $Z$ satisfy Property~(e).
This concludes the proof.
\end{proof}


\begin{proof}[Proof of Theorem~\ref{T:GLN}]
By Corollary~\ref{C:lnp abstract vanishing}, it suffices to show that, when Conditions~\eqref{C:sigma>nu} and~\eqref{C:dist} are satisfied, there exist an admissible rescaling function $\rho$, a compatible potential $\phi$ and positive constants $\omega$ and $c$ such that
\begin{equation}\label{E:Phi_L>c}
    \Phi_{\rho,\phi}^{\omega,\nu}(x)\geq c,\qquad x\in X^\rm o,
\end{equation}
where $\Phi_{\rho,\phi}^{\omega,\nu}$ is the function defined by~\eqref{E:Phi}.

Using Condition~\eqref{C:dist}, pick a constant $\Lambda$ satisfying
\begin{equation}
    \frac{\pi}{\sqrt{\bar n\sigma}}<\Lambda<\dist\bigl(K,\partial X\bigr).
\end{equation}
Choose a constant $\omega>0$ and smooth functions $Z\colon (0,\infty)\to (0,1]$ and $Y\colon(0,\infty)\to [0,\infty)$ satisfying Conditions~(a)--(e) of Lemma~\ref{L:diff inequality}.

Let $\tau$ be the distance function from the boundary of $X$.
We regard $\tau$ as a function on $X^\rm o$.
Fix a constant $\delta>0$ and consider the function 
\[
    \mu(x)\defeq\min\bigl({\tau(x)}/{2},\delta\bigr),\qquad x\in X^\rm o.
\]
By~\cite[Proposition~2.1]{GW79}, there exists a smooth function $\tau_\delta$ on $X^\rm o$ such that 
\begin{equation}\label{C:tau_delta}
    \abs{\tau(x)-\tau_\delta(x)}<\mu(x)\qquad\forall x\in X^\rm o\qquad\text{and}\qquad \norm{\dd\tau_\delta}_\infty<1+\delta.
\end{equation}
Observe that these conditions imply that $\tau_\delta$ is positive and goes to $0$ at infinity.
Choose a constant $\Lambda_1$ such that $\Lambda<\Lambda_1<\dist(K,\partial X)$.
Consider the open sets
\begin{equation}
    \Omega_0\defeq\left\{x\in X^\rm o\mid \tau_\delta(x)>\Lambda\right\}\qquad\text{and}\qquad
    \Omega_1\defeq\left\{x\in X^\rm o\mid \tau_\delta(x)<\Lambda_1\right\}.
\end{equation}
Observe that $X^\rm o=\Omega_0\cup\Omega_1$ and that, by taking $\delta$ small enough, $K\subset \Omega_0$ and $\Omega_1\subset X^\rm o\setminus K$.

Define functions $\rho\colon X^\rm o\to(0,1]$ and $\phi\colon X^\rm o\to[0,\infty)$ by setting
\begin{equation}
    \rho\defeq Z\circ\tau_\delta\qquad\text{and}\qquad\phi\defeq Y\circ\tau_\delta.
\end{equation}
By Proposition~\ref{P:admissible criterion}, Remark~\ref{R:coinciding at infty} and Property~(a) of Lemma~\ref{L:diff inequality}, $\rho$ is an admissible rescaling function.
Moreover, by Property~(b) of Lemma~\ref{L:diff inequality}, we have
\begin{enumerate}[label=(\roman*)]
    \item $\rho=\rho_0$ on $\Omega_0$, for a constant $\rho_0>0$.
\end{enumerate}
By Properties~(c) and~(d) of Lemma~\ref{L:diff inequality}, $\phi$ is a compatible potential and satisfies
\begin{enumerate}[resume*]
    \item $\phi=0$ on $\Omega_0$.
\end{enumerate}
Using~(i) and~(ii), we deduce
\begin{equation}\label{E:Omega_0}
    \Phi_{\rho,\phi}^{\omega,\nu}(x)=
    \frac{\bar n\omega}{1+\omega}\rho_0^4\nscal_g(x)
    -\frac{\bar n\omega}{1+\omega}\rho_0^4\nu(x),\qquad x\in\Omega_0.
\end{equation}
Using Condition~\eqref{C:sigma>nu} and Condition~\eqref{C:scal>sigma}, from~\eqref{E:Omega_0} we obtain
\[
    \Phi_{\rho,\phi}^{\omega,\nu}(x)\geq
    \frac{\rho_0^4\bar n \omega}{1+\omega}\inf_{y\in K}\bigl(\nscal_g(y)-\nu(y)\bigr)>0,\qquad x\in K
\]
and
\[
    \Phi_{\rho,\phi}^{\omega,\nu}(x)\geq\frac{\bar n\bar\sigma\omega}{1+\omega}\rho_0^4>0,\qquad x\in\Omega_0\setminus K.
\]
Therefore, Inequality~\eqref{E:Phi_L>c} holds on $\Omega_0$.

In order to complete the proof, we will show that, for $\delta$ small enough, Inequality~\eqref{E:Phi_L>c} holds on $\Omega_1$ as well.
Using~\eqref{C:tau_delta}, we have
\begin{equation*}
    \abs{\dd\rho_x}\leq(1+\delta)\abs{Z^\prime(\tau_\delta(x))}\qquad\text{and}\qquad
    \abs{\dd\phi_x}\leq(1+\delta)\abs{Y^\prime(\tau_\delta(x))}.
\end{equation*}
Since $\Omega_1\subset X^\rm o\setminus K$, $\scal_g\geq\sigma$ on $\Omega_1$.
Since $\nu$ is a $K$-bounding function, $\nu=0$ on $\Omega_1$.
By taking $\delta$ small enough and using Property~(e) of Lemma~\ref{L:diff inequality}, we deduce that there exists a constant $c_1>0$ such that
\begin{multline*}
    \Phi_{\rho,\phi}^{\omega,\nu}(x)\geq
    \frac{\bar n\bar\sigma \omega}{1+\omega}\rho^4(x)
    -\omega \rho^2(x)\abs{\dd\rho_x}^2+\phi^2(x)-\rho^2(x) \abs{\dd\phi_x}\\
    \geq\frac{\bar n\bar\sigma \omega }{1+\omega }Z^4(\tau_\delta(x))
    -\omega Z^2(\tau_\delta(x))(1+\delta)^2\abs{Z^\prime(\tau_\delta(x))}^2\\
    +Y^2(\tau_\delta(x))
    -Z^2(\tau_\delta(x)) (1+\delta)\abs{Y^\prime(\tau_\delta(x))}\geq c_1
\end{multline*}
for every $x\in\Omega_1$.
This concludes the proof.
\end{proof}


\subsection{Proof of Theorem~\ref{T:LNP}}\label{SS:LNP}
Assume first that $n$ is even.
Let $(E,\nabla^E)$ and $(F,\nabla^F)$ be the Hermitian bundles with metric connections constructed in Example~\ref{E:GL relative} using the map $f$.
Recall that $(E,\nabla^E)$ and $(F,\nabla^F)$ satisfy Assumption~\eqref{A:GL}, with $\supp(\dd f)$ an essential support of $(E,F)$, and
\begin{texteqn}\label{E:ALN4}
    $\relind(X^\rm o;E,F)= 0$ implies $\deg(f)=0$.
\end{texteqn}
Moreover, ${\cal R}^F=0$ everywhere, ${\cal R}^E=0$ on $M\setminus\supp(\dd f)$ and, since $f$ is strictly area decreasing, ${\cal R}^E\geq -\theta n(n-1)/4$ on $\supp(\dd f)$, for some $\theta\in(0,1)$.
By Condition~\eqref{E:ALN1}, there exists a compact subset $K\subset X^\rm o$ such that $\supp(\dd f)\subset K^\rm o$ and
\begin{equation}
    \scal_g(x)\geq\theta_1n(n-1),\qquad x\in K
\end{equation}
for some constant $\theta_1\in(\theta,1)$.
Observe that $K$ is an essential support of $(E,F)$.
Let $\nu\colon X^\rm o\to [0,\theta n(n-1)/4]$ be a smooth function such that $\nu=\theta n(n-1)/4$ on $\supp(\dd f)$ and $\nu=0$ on $X^\rm o\setminus K$.
Then $\nu$ is a $K$-bounding function and
\begin{equation*}\label{E:ALN3}
    \frac{\scal_g}{4}\geq\frac{\theta_1n(n-1)}{4}>\frac{\theta n(n-1)}{4}\geq\nu\qquad \text{on } K.
\end{equation*}
The last inequality,~\eqref{E:ALN2} and Theorem~\ref{T:GLN} imply that $\relind(X^\rm o;E,F)= 0$.
The thesis now follows from~\eqref{E:ALN4}.

Suppose now that $n$ is odd and the map $f$ is constant in a neighborhood of $\partial X$.
As in the proof of~\cite[Proposition~6.10]{LM89}, we fix a $1$-contracting map $\mu\colon S^n\times S^1\to S^{n+1}$ of degree one, which is constant on $\{\ast\}\times S^n\cup S^1\times \{\ast^\prime\}$.
Here, $\ast$ and $\ast^\prime$ are distinguished points respectively in $S^n$ and $S^1$ with $f(\partial X)\subset\{\ast\}$.
For $R>0$, let $S^1_R$ be the circle of radius $R$ and consider the manifold $X_1\defeq X\times S^1_R$ equipped with the product metric, denoted by $g_1$.
Let $f_1\colon X_1\to S^{n+1}$ be the map given by the composition $\mu\circ(f\times 1/R)$.
Then $f_1$ is area decreasing for $R$ large enough.
Moreover, $\scal_{g_1}=\scal_g\geq\sigma$ and $\dist(\supp(\dd f_1),\partial X_1)=\dist(\supp(\dd f),\partial X)$.
Now the thesis follows from the even dimensional case.
\qed


\subsection{Proof of Theorem~\ref{T:HLN}}\label{SS:HLN}
Let $g$ be a Riemannian metric on $X$ whose scalar curvature is bounded from below by a constant $\sigma>0$.
Consider the incomplete Riemannian manifold $X^\rm o$.
Let $(E,\nabla^E)$ and $(F,\nabla^F)$ be the flat bundles on $X^\rm o$ with typical fiber $C^\ast\Gamma$ constructed in Example~\ref{E:higher relative}.
Recall that, since $Y$ satisfies Condition~\eqref{C:relative higher index}, the class $\relind(X^\rm o;E,F)\in\KO_n(C^\ast\Gamma)$ does not vanish.
Moreover, for $0<R<\rad^\odot_g(\partial X)$, the geodesic collar neighborhoods $B_R(S^{n-1}_1),\ldots,B_R(S^{n-1}_N)$ are pairwise disjoint and the closure of the set
\[
    X\setminus \bigsqcup_{i=1}^NB_R(S^{n-1}_i)
\]
is an essential support of $(E,F)$ that we denote by $K_R$.
Since $\dist(K_R,\partial X)=R$ and $\relind(X^\rm o;E,F)\neq 0$, by Theorem~\ref{T:GLN} we deduce
\[
    R\leq\pi\sqrt{\frac{n-1}{n\sigma}},\qquad 0<R<\rad^\odot_g(\partial X)
\]
from which Inequality~\eqref{E:HLN1} follows.


\subsection{Proof of Theorem~\ref{T:complete metrics on torus}}\label{SS:proof of Theorem C}
Suppose there exists a complete Riemannian metric $g$ on $M$ such that $\scal_g>0$ everywhere.
Let $(E,\nabla^E)$ and $(F,\nabla^F)$ be the flat bundles on $M$ with typical fiber $C^\ast\Gamma$ constructed in Example~\ref{E:higher relative}.
Since $Y$ satisfies Condition~\eqref{C:relative higher index}, we have
\begin{equation}\label{E:HLN2}
    \relind(M;E,F)\neq 0.
\end{equation}
In order to obtain a contradiction, we will construct a generalized Gromov-Lawson operator on $M$.

Since the metric $g$ is complete, the function $\rho=1$ is admissible.
Let $K$ be an essential support of $(E,F)$ and let $\phi\colon M\to [0,\infty)$ be a smooth function such that $\phi=0$ in a neighborhood of $K$ and $\phi=1$ in a neighborhood of infinity.
Observe that, for $\lambda>0$, $\lambda\phi$ is a compatible potential.
Denote by $\Pop_\lambda$ the generalized Gromov-Lawson operator $\Pop_{1,\lambda\phi}$.
Let $\Phi_{\lambda}\colon M\to\RR$ be the smooth function defined by the formula
\begin{equation}\label{E:Phi_lambda}
    \Phi_\lambda(x)\defeq \frac{1}{4}\scal_g(x)+\lambda^2\phi^2(x)-\lambda\abs{\dd\phi_x},\qquad x\in M.
\end{equation}
Since the connections $\nabla^E$ and $\nabla^F$ are flat, from Lemma~\ref{L:Fredholm} and the Lichnerowicz formula~\eqref{E:Lichnerowicz} we deduce
\begin{equation}\label{T:HLN1}
    \bigl<\Pop^2_\lambda w,w\bigr>\geq\bigl<\Phi_\lambda w,w\bigr>,\qquad w\in\Gamma_\rm c(M;W).
\end{equation}
In order to prove the thesis, we will show that the function $\Phi_\lambda$ is uniformly positive for $\lambda$ small enough.
In fact, in this case~\eqref{T:HLN1} and Theorem~\ref{T:index} imply $\relind(M;E,F)=0$, contradicting~\eqref{E:HLN2}.

Let $\Omega_0$ and $\Omega_1$ be open subsets of $M$ such that $M=\Omega_0\cup\Omega_1$, $\Omega_0$ is relatively compact, and $\phi=1$ on $\Omega_1$.
Since $\Omega_0$ is relatively compact, there exists a constant $\sigma>0$ such that $\scal_g\geq\sigma$ on $\Omega_0$.
Choose $\lambda$ satisfying
\begin{equation}\label{T:HLN3}
    0<\lambda<\frac{\sigma}{4\norm{\dd\phi}_\infty}.
\end{equation}
With this choice, we have
\begin{equation}
    \Phi_\lambda(x)\geq \frac{\sigma}{4}-\lambda\norm{\dd\phi}_\infty>0,\qquad x\in \Omega_0.
\end{equation}
Since $\phi=1$ on $\Omega_1$, we also have
\begin{equation}
    \Phi_\lambda(x)\geq \lambda^2,\qquad x\in \Omega_1.
\end{equation}
Therefore, when $\lambda$ satisfies~\eqref{T:HLN3}, the function $\Phi_\lambda$ is uniformly positive.\qed


\section{Estimates on band widths}\label{S:BW}
This last section is devoted to the proof of Theorem~\ref{T:sharp band width}.
In Subsection~\ref{SS:rescaled Callias}, using rescaling functions and potentials in a similar fashion as in Subsection~\ref{S:GL}, we extend the theory of Callias-type operators to Riemannian manifolds which are not necessarily complete.
More precisely, we develop a rescaled version of the real Callias-type operators used by Zeidler in~\cite{Zei19}.
In Subsection~\ref{SS:Callias on spin bands}, we focus on compact Riemannian spin bands and prove a Callias-type index theorem, stating that the index of a Callias-type operator on a compact Riemannian spin band coincides with the index of an elliptic differential operator on a separating hypersurface.
Finally, in Subsection~\ref{SS:Callias vanishing} we prove a vanishing theorem yielding Theorem~\ref{T:sharp band width}.


\subsection{Generalized Callias-type operators}\label{SS:rescaled Callias}
Let $(M,g)$ be an $n$-dimensional Riemannian spin manifold.
Let $A$ be a Real unital $C^\ast$-algebra and let $\bigl(E,\nabla^E\bigr)$ be a bundle of finitely generated projective Hilbert $A$-modules endowed with a metric connection.
Let $\spinor_M$ be the associated $\cl_{n,0}$-linear spinor bundle with Dirac operator $\spind_M$.
Denote by $\Zop\colon\Gamma(M;\spinor_M\hat\tensor E)\to\Gamma(M;\spinor_M\hat\tensor E)$ the operator $\spind_M$ twisted with the bundle $E$.


\begin{definition}
We say that a smooth function $\psi\colon M\to\RR$ is a \emph{Callias potential} if there exist a compact subset $K\subset M$ and a constant $c>0$ such that $\psi^2-\abs{\dd\psi}>c$ on $M\setminus K$.
\end{definition}

Fix an admissible rescaling function $\rho$ and a Callias potential $\psi$.
Let $\Zop_\rho$ be the operator $\Zop$ rescaled with the function $\rho$ defined by Formula~\eqref{E:rescaled operator}.
The \emph{generalized Callias-type operator} associated to these data is the first order elliptic differential operator
$\Bop_{\rho,\psi}\colon\Gamma(M;\spinor_M\hat\tensor E\hat\tensor\cl_{0,1})\to \Gamma(M;\spinor_M\hat\tensor E\hat\tensor\cl_{0,1})$ defined as
\begin{equation}\label{E:gen Callias}
    \Bop_{\rho,\psi}\defeq \Zop_\rho\hat\tensor 1+\psi\hat\tensor\epsilon,
\end{equation}
where $\epsilon$ denotes left-multiplication by the Clifford generator of $\cl_{0,1}$.


\begin{remark}
When the metric $g$ is complete, $\rho=1$ is admissible.
If the Callias potential $\psi$ is a proper map with bounded gradient, $\Bop_{1,\psi}$ coincides with the operator used in~\cite{Zei19}.
\end{remark}


We now show that the operator $\Bop_{\rho,\psi}$ has a well-defined index.


\begin{theorem}\label{T:Callias Fredholm}
For every admissible rescaling function $\rho$ and every Callias potential $\psi$, the pair $(M,\Bop_{\rho,\psi})$ is complete and the operator $\Bop_{\rho,\psi}^2$ is uniformly positive at infinity.
\end{theorem}


\begin{proof}
The completeness of the pair $(M,\Bop_{\rho,\psi})$ follows from Proposition~\ref{P:adm->complete} and Remark~\ref{R:completeness depends on symbol}.
Moreover, we have
\begin{equation}\label{E:Callias square}
    \Bop_{\rho,\psi}^2=\Zop_\rho^2\hat\tensor 1+\rho^2\cc(\dd\psi)\hat\tensor\epsilon+\psi^2.
\end{equation}
Since $\rho\leq1$, from the previous identity we deduce
\begin{equation}\label{E:Callias vanishing}
    \left<\Bop_{\rho,\psi}^2w,w\right>\geq \left<\left(\psi^2-\abs{\dd\psi}\right)w,w\right>,\qquad w\in\Gamma_c(M;\spinor_M\hat\tensor E\hat\tensor\cl_{0,1}).
\end{equation}
Since $\psi$ is a Callias potential, the last inequality implies that $\Bop_{\rho,\psi}^2$ is uniformly positive at infinity.
\end{proof}


From Theorem~\ref{T:Callias Fredholm} and the results of Subsection~\ref{SS:relative completeness}, the class $\ind\left(\Bop_{\rho,\psi}\right)$ in $\KO_n(A)$ is well defined, for every admissible rescaling function $\rho$ and every Callias potential $\psi$.


\begin{remark}\label{E:Zeidler}
When $(M,g)$ is complete, $\rho=1$ and the Callias potential $\psi$ is a proper function with uniformly bounded gradient, the index of $\Bop_{1,\psi}$ coincides with the class $\ind_\rm{PM}\left(\spind_{M,E},\psi\right)$ used in~\cite{Zei19}.
\end{remark}


We conclude this subsection with some stability properties of the index of $\Bop_{\rho,\phi}$.


\begin{proposition}\label{P:Callias stability}
Suppose $\rho$ is an admissible rescaling function and $\psi$ is a Callias potential.
Then
\begin{enumerate}[label=$(\alph*)$]
    \item if $\rho^\prime$ is a second admissible rescaling function coinciding with $\rho$ outside of a compact set, then the indices of $\Bop_{\rho,\psi}$ and $\Bop_{\rho^\prime,\psi}$ coincide;
    \item if $\psi^\prime$ is a second Callias potential coinciding with $\psi$ outside of a compact set, then the indices of $\Bop_{\rho,\psi}$ and $\Bop_{\rho,\psi^\prime}$ coincide;
    \item if $\psi$ is constant and nonzero outside of a compact set, then the index of $\Bop_{\rho,\psi}$ vanishes.
\end{enumerate}
\end{proposition}


\begin{proof}
For Parts~(a) and~(b), it suffices to consider the linear homotopies $\rho_t\defeq t\rho^\prime+(1-t)\rho$ and $\psi_t=t\psi+(1-t)\psi$, with $0\leq t\leq1$, and argue as in the proof of Lemma~\ref{L:stability}.
Let us prove Part~(c).
Let $\psi_0$ be a constant nonzero function such that $\psi=\psi_0$ outside of a compact set.
Observe that $\psi_0$ is a Callias potential.
By Identity~\eqref{E:Callias square}, the operator $\Bop_{\rho,\psi_0}^2$ is uniformly positive and the index of $\Bop_{\rho,\psi_0}$ vanishes.
Using Part~(b), we deduce that the index of $\Bop_{\rho,\psi}$ vanishes as well. 
\end{proof}


\subsection{Callias-type operators on Riemannian spin bands}\label{SS:Callias on spin bands}
We start with recalling the notion of band due to Gromov~\cite[Section~2]{Gro18}.
A \emph{band} is a manifold $V$ with two distinguished subsets $\partial_\pm V$ of the boundary $\partial V$. 
It is called \emph{proper} if each $\partial_\pm V$ is a union of connected
components of the boundary and $\partial V=\partial_-V\sqcup\partial_+V$. 
If $V$ is a Riemannian
manifold, we define the width of $V$ as $\width(V)\defeq\dist(\partial_-V,\partial_+V)$.

In order to define a generalized Callias-type operator in this setting, we proceed as in Subsection~\ref{SS:ALNP} and consider the open manifold $V^\rm o\defeq V\setminus \partial V$.
The metric $g$ induces an incomplete metric on $V^\rm o$, that we denote by the same symbol.
Given a collar neighborhood $U_+\subset V$ of $\partial_+V$, we say that $U_+^\rm o\defeq U_+\setminus\partial_+V\subset V^\rm o$ is a neighborhood of the positive boundary at infinity of $V^\rm o$.
In a similar way, define the notion of a neighborhood of the negative boundary at infinity of $V^\rm o$.


\begin{definition}
A Callias potential $\psi$ on $(V^\rm o,g)$ is called \emph{band compatible} if there exist constants $\lambda_-$ and $\lambda_+$, with $\lambda_-<0<\lambda_+$, such that the image of $\psi$ is contained in $[\lambda_-,\lambda_+]$, $\psi=\lambda_-$ in a neighborhood of the negative boundary at infinity of $V^\rm o$ and $\psi=\lambda_+$ in a neighborhood of the positive boundary at infinity of $V^\rm o$.
\end{definition}


\noindent Let $\bigl(E,\nabla^E\bigr)$ be a bundle of finitely generated projective Real Hilbert $A$-modules with inner product and metric connection.
We now study the properties of the generalized Callias-type operator $\Bop_{\rho,\psi}$, where $\rho$ is an admissible rescaling function and $\psi$ is a band compatible Callias potential.


\begin{lemma}\label{L:Properties Callias boundary}
Let $\rho$ be an admissible rescaling function.
Suppose $\psi_1$ and $\psi_2$ are two band compatible Callias potentials.
Then the indices of $\Bop_{\rho,\psi_1}$ and $\Bop_{\rho,\psi_2}$ coincide.
\end{lemma}


\begin{proof}
Consider the linear homotopy $\psi_t=t\psi_1+(1-t)\psi_2$, with $0\leq t\leq1$.
Since $\psi_1$ and $\psi_2$ are band compatible, $\psi_t$ is a band compatible Callias potential for all $t\in[0,1]$.
The thesis follows by arguing as in the proof of Lemma~\ref{L:stability}.
\end{proof}


\begin{theorem}\label{T:Callias index}
Let $\rho$ be an admissible rescaling function and let $\psi\colon V^\rm o\to[\lambda_-,\lambda_+]$ be a band compatible Callias potential.
If $a\in(\lambda_-,\lambda_+)$ is a regular value of $\psi$, then
\[
    \ind \left(\Bop_{\rho,\phi}\right)
    =\ind\left(\spind_{\phi^{-1}(a),E|_{\phi^{-1}(a)}}\right)
    \in\KO_{n-1}(A).
\]
\end{theorem}


\begin{proof}
Let $U\cong \partial V\times[0,1)$ be a collar neighborhood of $\partial V$ such that $\psi$ is constant on $U\setminus\partial V$.
Using Part~(a) of Proposition~\ref{P:Callias stability}, assume there exists a constant $\rho_0\in(0,1]$ such that $\rho=\rho_0$ on the complement of $U^\prime\cong \partial V\times[0,1/4)$.
Observe that the manifold $N\cong\partial V\times\{1/2\}$ is a closed separating hypersurface of $V^\rm o$.
Moreover, $N=N_+\sqcup N_-$, where $N_+\simeq \partial_+V$ and $N_-\cong\partial_-V$.
Therefore, we have the partition
\[
    V^\rm o=Y\cup_{N}W.
\]
Here, $Y$ is a compact manifold with $\partial Y=N_+\sqcup N_-$ and $W=W_-\sqcup W_+$, where $W_-$ is a neighborhood of the negative boundary at infinity with $\partial W_-\cong N_-$ and $W_+$ is a neighborhood of the positive boundary at infinity with $\partial W_+\cong N_+$.
Let us assume (after deformation near $N$) that our data respect the product structure of a tubular neighborhood of $N$ where the function $\psi$ is constant.
Consider the half-cylinders
\[
    Z_-\defeq (-\infty,0]\times N_-\qquad\text{and}\qquad Z_+\defeq [0,\infty)\times N_+.
\]
Observe that the bundles $E|_{N_\pm}$ extend to bundles $E_{Z_\pm}$ with metric connections on $Z_\pm$.
Consider the manifolds
\[
    (M_2)_-\defeq Z_-\cup_{N_-}W_-^-\qquad (M_2)_+\defeq W_+^-\cup_{N_+}Z_+\qquad M_2\defeq (M_2)_-\sqcup (M_2)_+
\]
where $W_-^-$ and $W_+^-$ denote respectively the manifolds $W_-$ and $W_+$ with opposite orientations. 
Let $g_2$ be the Riemannian metric coinciding with $g$ on $W_\pm^-$ and being a product on the half cylinders $Z_\pm$.
Let $E_2$ be the bundle with metric connection coinciding with $E$ on $W_-^-\sqcup W_+^-$ and with $E_{Z_\pm}$ respectively on $Z_\pm$.
Let $\rho_2^+\colon M_2^+\to(0,1]$ be the smooth function coinciding with $\rho$ on $W_+^-$ and with $\rho_0$ on $Z_+$.
Observe that $\rho_2^+$ is admissible, since the metric $g_2$ is complete on the cylindrical end $Z_+$.
Finally, let $\mu_+$ be a positive constant such that $\lambda_+^2>\mu_+$.
Then there exists a smooth function $\psi_2^+\colon (M_2)_+\to[\lambda_+,\infty)$ such that 
\begin{enumerate}[label=(\roman*)]
    \item $\psi_2^+=\lambda_+$ on a neighborhood of $W_+^-$;
    \item $\psi_2^+(t,x)\geq\mu_+t$ for all $(t,x)\in [0,\infty)\times N_+\subset Z_+$;
    \item $\abs{\dd\psi_2^+}\leq\mu_+$.
\end{enumerate}
By Properties~(i) and~(ii), $\psi_2^+$ is a Callias potential.
Let $\Bop_{\rho_2^+,\psi_2^+}$ be the associated generalized Callias-type operator.
Since $\rho_2^+\leq 1$ and $\bigl(\psi_2^+\bigr)^2\geq (\lambda_+)^2>\mu_+$, Properties~(i)--(iii) imply that the function $\bigl(\psi_2^+\bigr)^2-\abs{\dd\psi_2^+}$ is uniformly positive on $(M_2)_+$.
By Inequality~\eqref{E:Callias vanishing}, we deduce that the index class of $\Bop_{\rho_2^+,\psi_2^+}$ vanishes.
Finally, observe that $\psi_2^+$ is proper with uniformly bounded gradient on the cylindrical end $Z_+$.

In a similar way, construct an admissible rescaling function $\rho_2^-$ and a Callias potential $\psi_2^-\colon (M_2)_-\to(-\infty,\lambda_-]$ such that the index of the associated generalized Callias-type operator $\Bop_{\rho_2^-,\psi_2^-}$ vanishes and the function $\psi_2^-$ is proper with uniformly bounded gradient on the cylindrical end $Z_-$.
Finally, observe that $\rho_2^\pm$ induce an admissible rescaling function $\rho_2$ on $M_2$ and $\psi_2^\pm$ induce a Callias potential $\psi_2$ on $M_2$.
Since $M_2=M_2^-\sqcup M_2^+$, the index of the associated operator $\Bop_{\rho_2,\psi_2}$ vanishes.

Observe that the manifolds $V^\rm o$ and $M_2$ satisfy Assumption~\ref{A:cut-and-paste}.
Using the cut-and-paste construction described in Subsection~\ref{SS:cut-and-paste}, we obtain Riemannian manifolds
\[
    M_3\defeq Z_-\cup_{N_-}Y\cup_{N_+}Z_+\qquad\text{and}\qquad
    M_4\defeq W\cup_NW^-
\]
and generalized Callias-type operators $\Bop_{\rho_3,\psi_3}$ and $\Bop_{\rho_4,\psi_4}$ respectively on $M_3$ and $M_4$.
Notice that the potential $\psi_4$ is constant and nonzero on both connected components of $\Bop_{\rho_4,\psi_4}$.
Therefore, the index of $\Bop_{\rho_4,\psi_4}$ vanishes by Part~(c) of Proposition~\ref{P:Callias stability}.

Let us now analyze the operator $\Bop_{\rho_3,\psi_3}$.
By construction, $\psi_3$ is a proper smooth function whose gradient is uniformly bounded.
Moreover, since $a\in (\lambda_-,\lambda_+)$, $\psi_3^{-1}(a)=\psi^{-1}(a)$.
Using Remark~\ref{E:Zeidler} and~\cite[Theorem~A.1]{Zei19}, the indices of $\Bop_{\rho_3,\psi_3}$ and $\spind_{\phi^{-1}(a),E|_{\phi^{-1}(a)}}$ coincide.
Therefore, using Theorem~\ref{T:relative}, we obtain
\begin{multline*}
    \ind\left(\Bop_{\rho,\psi}\right)=\ind\left(\Bop_{\rho,\psi}\right)+\ind\left(\Bop_{\rho_2,\psi_2}\right)\\
    =\ind\left(\Bop_{\rho_3,\psi_3}\right)+\ind\left(\Bop_{\rho_4,\psi_4}\right)
    =\ind\left(\spind_{\phi^{-1}(a),E|_{\phi^{-1}(a)}}\right),
\end{multline*}
which concludes the proof.
\end{proof}


We finally specialize to the case when $(V,g)$ is a Riemannian band over a closed spin manifold $N$, i.e. $V$ is diffeomorphic to $N\times[-1,1]$.
Let $\bigl(\mathcal{L}_V,\nabla^{\mathcal{L}_V}\bigr)$ be the Mishchenko bundle of $V$ endowed with the canonical flat connection.
For an admissible rescaling function $\rho$ and a band compatible Callias potential $\psi$, denote by $\Bop_{\rho,\psi}^{\mathcal{L}_V}$ the generalized Callias-type operator associated to these data.


\begin{corollary}\label{C:Callias Mishchenko}
Let $N$ be a closed $(n-1)$-dimensional spin manifold with fundamental group $\Gamma$.
Let $(V,g)$ be a Riemannian spin band over $N$.
For an admissible rescaling function $\rho$ and a band compatible Callias potential $\psi$, we have
\begin{equation}\label{E:Callias mishchenko}
    \ind\left(\Bop_{\rho,\psi}^{\mathcal{L}_V}\right)=\alpha(N)    \in\KO_{n-1}(C^\ast\Gamma).
\end{equation}
\end{corollary}


\begin{proof}
For $\lambda>0$ small enough, let $\psi_1\colon V\to [-\lambda,\lambda]$ be a smooth function such that $0$ is a regular value, $\psi_1^{-1}(0)=N\times\{0\}$, the support of $\dd\psi_1$ is contained in the interior of a closed geodesic tubular neighborhood of $N\times\{0\}$, $\psi_1=\lambda$ in a neighborhood of $\partial_+V$, and $\psi_1=-\lambda$ in a neighborhood of $\partial_-V$.
Observe that $\psi_1$ is a band compatible Callias potential.
Since the inclusion $N\hookrightarrow V$ induces an isomorphism on $\pi_1$, Identity~\eqref{E:Callias mishchenko} follows from Lemma~\ref{L:Properties Callias boundary} and Theorem~\ref{T:Callias index}.
\end{proof}


\subsection{Proof of Theorem~\ref{T:sharp band width}}\label{SS:Callias vanishing}
Suppose
\begin{equation}\label{E:band width>c}
    \width(V)>2\pi\sqrt{\frac{n-1}{\sigma n}}.
\end{equation}
In order to prove the thesis, we need to show that $\alpha(N)=0$.

For an admissible rescaling function $\rho$ and a band compatible Callias potential $\psi$, consider the generalized Callias-type operator $\Bop_{\rho,\psi}^{\mathcal{L}_V}$ used in Corollary~\ref{C:Callias Mishchenko}.
Let $\Psi_{\rho,\phi}^{\omega}\colon V^\rm o\to\RR$ be the smooth function defined by the formula
\begin{equation}\label{E:Psi}
    \Psi_{\rho,\phi}^{\omega}(x)\defeq 
    \frac{\bar n\bar\sigma\omega}{1+\omega}\rho^4(x)
    -\omega\rho^2(x)\abs{\dd\rho_x}^2+\psi^2(x)-\rho^2(x) \abs{\dd\psi_x},\qquad x\in V^\rm o.
\end{equation}
Here, we use the notation $\bar n=n/(n-1)$ and $\bar\sigma=\sigma/4$ introduced respectively in Section~\ref{S:admissible rescaling} and Section~\ref{S:LNP}.
Since the bundle $\mathcal{L}_V$ is flat and $\scal_g\geq\sigma$, from Proposition~\ref{P:rescaled Lichnerowitcz} and Identity~\eqref{E:Callias square} we deduce
\begin{equation}\label{E:B_r,p>Psi}
    \Bigl<\bigl(\Bop_{\rho,\phi}^{\mathcal{L}_V}\bigr)^2w,w\Bigr>\geq\bigl<\Psi_{\rho,\phi}^{\omega}w,w\bigr>,
\end{equation}
for every $w\in\Gamma_\rm c(V^\rm o;\spinor_M\hat\tensor \mathcal{L}_V\hat\tensor\cl_{0,1})$ and every $\omega>0$.
By Corollary~\ref{C:Callias Mishchenko} and Inequality~\eqref{E:B_r,p>Psi}, it suffices to show that there exist an admissible rescaling function $\rho$, a band compatible Callias potential $\psi$ and positive constants $\omega$, $c$ such that
\begin{equation}\label{E:Psi>c}
    \Psi_{\rho,\psi}^\omega(x)\geq c,\qquad x\in V^\rm o.
\end{equation}
To this end, we will use Lemma~\ref{L:diff inequality} in a similar way as in the proof of Theorem~\ref{T:GLN}.

Using Condition~\eqref{E:band width>c}, choose a constant $\Lambda$ satisfying
\begin{equation}
    \frac{2\pi}{\sqrt{\bar n\sigma}}<2\Lambda<\width(V).
\end{equation}
Pick a constant $\omega>0$ and smooth functions $Z\colon (0,\infty)\to (0,\infty)$ and $Y\colon(0,\infty)\to [0,\infty)$ satisfying Conditions~(a)--(e) of Lemma~\ref{L:diff inequality}.

Let $\tau^+$ be the distance function from $\partial_+V$ and let $\tau^-$ be the distance function from $\partial_-V$.
We regard $\tau^\pm$ as functions on $V^\rm o$.
Fix a constant $\delta>0$ and consider the functions $\mu^+$ and $\mu^-$ defined as 
\[
    \mu^\pm(x)\defeq\min\bigl({\tau^\pm(x)}/{2},\delta\bigr),\qquad x\in V^\rm o.
\]
By~\cite[Proposition~2.1]{GW79}, there exist smooth functions $\tau_\delta^+$ and $\tau_\delta^-$ on $V^\rm o$ such that 
\begin{equation}\label{C:tau_delta^pm1}
    \abs{\tau^\pm(x)-\tau_\delta^\pm(x)}<\mu^\pm(x),\qquad\forall x\in V^\rm o
\end{equation}
and
\begin{equation}\label{C:tau_delta^pm2}
    \norm{\dd\tau_\delta^\pm}_\infty\leq1+\delta.    
\end{equation}
Choose constants $\Lambda_1$, $\Lambda_2$ such that $\Lambda<\Lambda_1<\Lambda_2<\width(V)/2$ and consider the open sets
\[
    \Omega_\pm\defeq\left\{x\in V^\rm o\mid \tau_\delta^\pm(x)<\Lambda_2\right\}.
\]
By taking $\delta$ small enough, $\bar\Omega_-\cap\bar\Omega_+=\emptyset$.
Define the set
\[
    \Omega\defeq\left\{x\in V^\rm o\mid \min\left(\tau_\delta^+(x),\tau_\delta^-(x)\right)>\Lambda_1\right\}.
\]
Let $\tau_\delta$ be a smooth function on $V^\rm o$ such that $\tau_\delta=\tau_\delta^-$ on $\Omega_-$, $\tau_\delta=\tau_\delta^+$ in $\Omega_+$ and $\tau_\delta(x)\geq\Lambda_1$ when $x\in\Omega$.
Observe that Properties~\eqref{C:tau_delta^pm1} and~\eqref{C:tau_delta^pm2} imply that $\tau_\delta$ is positive and goes to $0$ at infinity.

Define the smooth function $\rho\colon V^\rm o\to(0,1]$ by setting
\begin{equation}\label{E:rho Callias1}
    \rho\defeq Z\circ \tau_\delta.
\end{equation}
By Proposition~\ref{P:admissible criterion}, Remark~\ref{R:coinciding at infty} and Property~(a) of Lemma~\ref{L:diff inequality}, $\rho$ is an admissible rescaling function.
Moreover, by Property~(b) of Lemma~\ref{L:diff inequality}, we have
\begin{texteqn}\label{E:rho Callias2}
    $\rho=\rho_0$ on $\Omega$, for some constant $\rho_0\in(0,1]$.
\end{texteqn}
Define a function $\psi\colon V^\rm o\to[0,\infty)$ by setting
\begin{equation}\label{E:psi =Y tau_delta}
    \psi=\pm Y\circ \tau_\delta\qquad\text{on}\ \ \Omega_\pm\qquad\text{and}\qquad\psi=0\qquad\text{on}\ \ \Omega.
\end{equation}
By Properties~(c) and~(d) of Lemma~\ref{L:diff inequality}, $\psi$ is well defined and is a band compatible Callias potential.
Using~\eqref{E:rho Callias2} and~\eqref{E:psi =Y tau_delta}, we deduce
\begin{equation}
    \Psi_{\rho,\psi}^{\omega}(x)=    
    \frac{\bar n\bar\sigma\omega}{1+\omega}\rho_0^4>0,\qquad x\in\Omega.
\end{equation}
Therefore, Inequality~\eqref{E:Psi>c} holds on $\Omega$.

In order to complete the proof, we will show that, for $\delta$ small enough, Inequality~\eqref{E:Psi>c} holds on $\Omega_+$ and $\Omega_-$ as well.
Using~\eqref{C:tau_delta^pm2},~\eqref{E:rho Callias1} and~\eqref{E:psi =Y tau_delta} we have
\begin{equation*}
    \abs{\dd\rho_x}\leq(1+\delta)\abs{Z^\prime(\tau_\delta(x))}\qquad\text{and}\qquad
    \abs{\dd\phi_x}\leq(1+\delta)\abs{Y^\prime(\tau_\delta(x))}
\end{equation*}
for every $x\in\Omega_-\cup\Omega_+$.
By taking $\delta$ small enough and using Property~(e) of Lemma~\ref{L:diff inequality}, we deduce that there exists a constant $c_1>0$ such that
\begin{multline*}
    \Psi_{\rho,\phi}^{\omega}(x)
    = \frac{\bar n\bar\sigma \omega}{1+\omega}\rho^4(x)
    -\omega \rho^2(x)\abs{\dd\rho_x}^2+\psi^2(x)-\rho^2(x) \abs{\dd\psi_x}\\
    \geq\frac{\bar n\bar\sigma \omega }{1+\omega }Z^4(\tau_\delta(x))
    -\omega Z^2(\tau_\delta(x))(1+\delta)^2\abs{Z^\prime(\tau_\delta(x))}^2\\
    +Y^2(\tau_\delta(x))
    -Z^2(\tau_\delta(x)) (1+\delta)\abs{Y^\prime(\tau_\delta(x))}\geq c_1
\end{multline*}
for every $x\in\Omega_-\cup\Omega_+$.
This concludes the proof.
\qed

\end{document}